\documentclass[12pt]{article}

\usepackage{latexsym}
\usepackage{amsmath}
\usepackage{amssymb}
\usepackage{amsthm}
\usepackage{amscd}
\usepackage[mathscr]{eucal}
\usepackage{fullpage}
\usepackage{graphicx}
\usepackage{subfigure}
\usepackage{psfrag}
\usepackage{rotating}
\usepackage{appendix}
\usepackage[all]{xy}
\usepackage{multirow}

\usepackage{color}

\newcommand{\R}{\mathbb{R}}

\newcommand{\prad}{\widehat{p}_\mathrm{rad}}
\newcommand{\ptan}{\widehat{p}_\mathrm{tan}}
\newcommand{\ptilderad}{\widetilde{p}_\mathrm{rad}}
\newcommand{\ptildetan}{\widetilde{p}_\mathrm{tan}}
\newcommand{\textfrac}[2]{{\textstyle \frac{#1}{#2}}}

\newcommand{\V}[1]{\mathbf{#1}}

\newtheorem{theorem}{Theorem}

\newtheorem{lemma}{Lemma}
\newtheorem{proposition}{Proposition}

\newtheorem{definition}{Definition}

\usepackage{fullpage}
\numberwithin{equation}{section}

\title{\bf On self-gravitating\\ polytropic elastic balls
}
 \author
 {
       Simone Calogero  \\
       {\small Department of Mathematical Sciences}  \\
       {\small Chalmers University of Technology, University of Gothenburg}  \\
       {\small Gothenburg, Sweden} \\
       {\small  \tt  calogero@chalmers.se}
       }
\date{}
\begin{document}

\maketitle
\begin{abstract}
A new four-parameters family of constitutive functions for spherically symmetric elastic bodies is introduced 
which extends the two-parameters class of polytropic fluid models widely used in several applications of fluid mechanics. The four parameters in the polytropic elastic model are the polytropic exponent $\gamma$,  the bulk modulus $\kappa$, the shear exponent $\beta$ and the Poisson ratio $\nu\in (-1,1/2]$. The two-parameters class of polytropic fluid models arises as a special case when $\nu=1/2$ and $\beta=\gamma$.  In contrast to the standard Lagrangian approach to elasticity theory, 
the polytropic elastic model in this paper is formulated directly in physical space, i.e., in terms of Eulerian state variables, which is particularly useful for the applications e.g. to astrophysics where the reference state of the bodies of interest (stars, planets, etc.) is not observable. 
After discussing some general properties of the polytropic elastic model, the steady states and the homologous motion of Newtonian self-gravitating polytropic elastic balls are investigated. It is shown numerically that static balls exist when the parameters $\gamma,\beta$ are contained in a particular region $\mathcal{O}$ of the plane, depending on $\nu$, and proved analytically for $(\gamma,\beta)\in\mathcal{V}$, where $\mathcal{V}\subset\mathcal{O}$ is a disconnected set which also depends on the Poisson ratio $\nu$. Homologous solutions describing continuously collapsing balls are constructed numerically when $\gamma=4/3$. The radius of these solutions shrinks to zero in finite time, causing the formation of a center singularity with infinite density and pressure. Expanding self-gravitating homologous elastic balls are also constructed analytically for some special values of the shear parameter $\beta$.
\end{abstract}
\newpage

\section{Introduction}
The problem of self-gravitating bodies in astrophysics is the source of many interesting questions in mathematical physics.
An important example 
is the motion of self-gravitating fluids described by the Euler-Poisson system
\begin{align*}
&\partial_t\rho+\nabla\cdot(\rho \V{u})=0,\\
&\rho(\partial_t\V{u}+\V{u}\cdot \nabla \V{u})+\nabla p=-\rho\nabla \phi,\\
&\Delta\phi=4\pi G\rho,\quad \lim_{|x|\to\infty}\phi(t,x)=0.
\end{align*}
Here $\rho, \V{u},p$ are respectively the mass density, the velocity field and the pressure of the fluid, $\phi$ denotes the fluid own self-generated Newtonian gravitational field  and $G$ is Newton's gravitational constant. The system is completed by adding a constraint between the variables $\rho,\V{u},p$, e.g., $p=F(\rho)$ for barotropic fluids.
When applied to gaseous stars, the polytropic equation of state
\begin{equation}\label{polyfluid}
p=c\rho^\gamma, \quad c,\gamma>0,
\end{equation} 
is a standard choice in astrophysics~\cite{Ch} and remains today the most studied stellar model in the literature~\cite{KW}. 
%
%

The present paper is part of a series which aims to extend the theory of self-gravitating fluids to self-gravitating elastic bodies. The articles~\cite{AC18,AC19,ACA} focus on self-gravitating spherically symmetric elastic bodies in static equilibrium. In the present paper some first results on the problem of self-gravitating elastic balls in motion will be established. 

Even though elastic matter models have long been used in astrophysics~\cite{Jeans,Lord,Love}, and despite important applications to e.g.~describe the deformation of planets~\cite{MW16} and neutron stars crusts~\cite{CH}, the problem of self-gravitating elastic bodies has received so far scarce attention compared to the analogous problem for fluid bodies. This could in part be the result of an apparent conceptual discrepancy between (non-linear) elasticity and astrophysics:  the former is originally a theory in Lagrangian form~\cite{Ciarlet, MH, ogden} describing the state of bodies in terms of spatial deformations from a given reference configuration, whereas the latter is concerned with the properties of planets, stars, galaxies, etc., in their current (physical) state.
While it is possible to study the problem of self-gravitating elastic bodies in the Lagrangian formulation of continuum mechanics~\cite{ABS,BS1,BS2,CT}, for astrophysical applications the obtained results should eventually be translated in terms of Eulerian state variables, such as $\rho,\V{u},p$ in the fluid case, which can be determined by model-fitting observational data.
However this transition from the reference to the physical space is in general not possible because the deformation map solution of the field equations in Lagrangian coordinates is typically neither globally injective nor a very regular function. 

In~\cite{AC18} it was shown that these two different approaches of astrophysics and elasticity can be reconciled, namely that the equation of state of spherically symmetric elastic bodies can be defined directly in physical space as a constraint between Eulerian state variables. This formulation has been used to prove the existence of static self-gravitating elastic matter distributions with arbitrarily large strain~\cite{AC18,AC19}, albeit only in the spherically symmetric case. This is the formulation of the theory of self-gravitating elastic balls used in the present paper. It is described, and expanded, in Section~\ref{sec2}. The (formally) equivalent Lagrangian formulation is discussed in Section~\ref{lagsec}.

Another reason for the hard interplay between astrophysics and elasticity theory is the absence of a systematic study on which elastic constitutive functions may be relevant for the applications to the problem of self-gravitating elastic bodies. Some classical elastic material models, such as the Saint-Venant-Kirchhoff model, have been used to describe planet deformations~\cite{MW16}. However, as shown in Section~\ref{constsec} below,  see also~\cite{Ciarlet}, for large deformations the Saint-Venant-Kirchhoff model---and others---violates important physical and mathematical properties, such as global hyperbolicity and the Baker-Ericksen inequality. 

One aim of this paper is to introduce a new elastic material model for spherically symmetric bodies that generalises the polytropic equation of state for fluids. More precisely, the new constitutive function implies the following properties:
%
(i) the equations of motion (conservation of mass and balance of momentum) are scale invariant and (ii) the boundary of isolated balls is subject to constant shear. While in the case of barotropic fluids these two properties uniquely characterize the polytropic equation of state, in the elastic case they still leave the freedom to choose an arbitrary function of the shear in the constitutive function. In Section~\ref{polysec} a well-motivated choice for this shear function is made, which results in a particular elastic constitutive function satisfying the properties (i) and (ii) above and which, in addition, is globally hyperbolic and verifies the Baker-Ericksen inequality. Each model in this class will be referred to as ``polytropic elastic model" and depends on four parameters $\kappa,\nu,\beta,\gamma$. The parameters $\gamma,\kappa$ correspond to the polytropic exponent and the bulk modulus, 
$\nu$ is the Poisson ratio of the elastic material ($\nu=1/2$ for fluid matter), while $\beta$ is a new parameter expressing the dependence of the elastic constitutive function on shear deformations. 

The last section of the paper is devoted to the analysis of self-gravitating polytropic elastic balls. In Section~\ref{staticsec}  static balls are constructed numerically when $\beta,\gamma$ are contained in a particular connected set $\mathcal{O}$ depending on the Poisson ratio $\nu$, see Figure~\ref{fig2}. A conjecture is put forward that $\mathcal{O}$ is the largest region in the $(\gamma,\beta)$ plane for which static self-gravitating polytropic elastic balls exist. In the fluid limit $\nu\to 1/2$, $\beta=\gamma$, the conjecture reduces to the well-known fact that static self-gravitating polytropic fluid balls exist if and only if $\gamma>6/5$. In Appendix a special case of the conjecture is proved analytically, namely that static self-gravitating polytropic elastic balls form when the parameters $\gamma,\beta$ are contained in a certain disconnected region $\mathcal{V}\subset\mathcal{O}$, depending on the Poisson ratio $\nu$.
Finally in Section~\ref{homosec} the homologous motion of self-gravitating polytropic elastic balls is investigated. Homologous solutions describe self-gravitating balls which, after a time $t_0$, are either continuously expanding or continuously collapsing and which, in the second case, develop a singularity at the center in finite time. In the fluid case (self-gravitating) homologous---and more general group invariant---solutions have been investigated for a long time, see for instance~\cite{GW, GHJS, Makino2, Ov,sideris}. For the purpose of the present investigation the articles~\cite{GW} and~\cite{Makino2}, where homologous collapsing self-gravitating fluid balls have been constructed respectively numerically and analitically, are particularly relevant. 
In this paper similar collapsing solutions are found numerically in the elastic case. Subsequently the existence of expanding elastic balls is proved analytically for some special values of the shear parameter $\beta$. 

\section{Eulerian representation of elastic balls}\label{sec2}
A spherically symmetric matter distribution is a quadruple $(\rho,p_\mathrm{rad},p_\mathrm{tan},u)$ of real-valued functions of time $t\in\R$ and radius $r>0$, where $\rho(t,r)$ is the mass density, $p_\mathrm{rad}(t,r), p_\mathrm{tan}(t,r)$ are the radial and tangential pressure and $u(t,r)$ is the radial component of the velocity field. 
The Cauchy stress tensor of spherically symmetric matter distributions is
\[
\sigma_{ij}(t,x)=\sigma_{\mathrm{rad}}(t,r)\frac{x_i x_j}{r^2}+\sigma_{\mathrm{tan}}(t,r)\left(\delta_{ij}-\frac{x_i x_j}{r^2}\right),
\]
where  $x=(x_1,x_2,x_3)\in\R^3$ are spatial Cartesian coordinates such that $r=|x|$ and 
\[
\sigma_\mathrm{rad}(t,r)=-p_\mathrm{rad}(t,r),\quad \sigma_\mathrm{tan}(t,r)=-p_\mathrm{tan}(t,r)
\]
are the radial and tangential stress in the matter interior. 

 
A spherically symmetric matter distribution $(\rho,p_\mathrm{rad},p_\mathrm{tan},u)$ is called a spherically symmetric body if its domain in the radial variable is, at all times, non-empty, connected and bounded. Any such body is therefore either a ball or a shell of matter. 
The focus of this paper is on balls of matter with strongly regular center, whose precise definition is the following. 

\begin{definition}\label{strongregball}
Let $T>0$. A quadruple $(\rho,p_\mathrm{rad},p_\mathrm{tan},u):[0,T]\times [0,\infty)\to\R^4$ is said to be a ball of matter with strongly regular center if there exists $R:[0,T]\to (0,\infty)$ such that $R\in C^1([0,T])$ and
\begin{itemize}
\item[(i)] $\rho(t,r), p_\mathrm{rad}(t,r)$ and $p_\mathrm{tan}(t,r)$ are positive for $r\in [0,R(t))$, for all $t\in [0,T]$;
\item[(ii)] $(\rho,p_\mathrm{rad},p_\mathrm{tan},u)\in C^0( \Omega_T)\cap C^1(\Omega^*_T)$, where 
\begin{align*}
&\Omega_T:=\{(t,r):0\leq t\leq T, 0\leq r\leq R(t)\},\\
&\Omega_T^*:=\{(t,r):0\leq t\leq T, 0\leq r< R(t)\}
\end{align*}
\item[(iii)] $\rho=p_\mathrm{rad}=p_\mathrm{tan}=u=0$, for $(t,r)\in [0,T]\times [0,\infty)\diagdown\Omega_T$;
\item[(iv)] $ p_\mathrm{rad}(t,0)=p_\mathrm{tan}(t,0)$,  $\partial_tp_\mathrm{rad}(t,0)=\partial_tp_\mathrm{tan}(t,0)$ and $u(t,0)=\partial_tu(t,0)=0$;
\item[(v)]
$\lim_{r\to 0^+}\partial_r\rho(t,r)=\lim_{r\to 0^+}\partial_rp_\mathrm{rad}(t,r)=\lim_{r\to 0^+}\partial_rp_\mathrm{tan}(t,r)=0$, 
and
\[
\lim_{r\to 0^+}\partial_{r}u(t,r)=\lim_{r\to 0^+}u(t,r)/r:=\omega(t),\quad \omega\in C([0,T]).
\]
\end{itemize}
\end{definition}
The fundamental equations describing the evolution of balls of matter are the continuity equation for the mass density and the balance of momentum equation:
\begin{subequations}\label{sssystem}
\begin{align}
&\partial_t\rho+\frac{1}{r^2}\partial_r(r^2\rho u)=0,\label{conteq}\\
&\rho(\partial_t u+u\partial_r u)=-\partial_r p_\mathrm{rad}+\frac{2}{r}(p_\mathrm{tan}-p_\mathrm{rad})+\rho f,
\end{align}
\end{subequations}
where $f(t,r)$ denotes the force (per unit of mass) acting on the ball and which is not due to the internal strain. In this paper we are mostly interested in self-gravitating balls, for which 
\begin{subequations}\label{Newtonforce}
\begin{equation}
f(t,r)=-G\frac{m}{r^2}
\end{equation}
where $G$ is Newton's gravitational constant and
\begin{equation}\label{mdef}
m(t,r)=4\pi\int_0^r\rho(t,s)\,s^2\,ds.
\end{equation}
\end{subequations}

The system~\eqref{sssystem} is posed in the matter interior and has to be supplemented by boundary conditions at $r=R(t)$. We assume
\begin{equation}\label{bccond}
u(t,R(t))=\dot{R}(t),\quad p_\mathrm{rad}(t,R(t))=0.
\end{equation}
The boundary condition $u(t,R(t))=\dot{R}(t)$ means that the boundary of the ball is comoving with the matter and implies, using~\eqref{conteq}, that the total mass of the ball 
\[
m(t,R(t))=4\pi\int_0^{R(t)}\rho(t,r) r^2\,dr
\]
is conserved. The boundary condition $p_\mathrm{rad}(t,R(t))=0$ signifies that the ball is surrounded by vacuum; see~\cite{KW} for other examples of boundary conditions used in astrophysics.

%

The system~\eqref{sssystem} is completed by adding an equation of state on the variables $(\rho,p_\mathrm{rad},p_\mathrm{tan},u)$.
As mentioned in the Introduction, a popular example is the barotropic fluid equation of state $p_\mathrm{rad}(t,r)=p_\mathrm{tan}(t,r)=F(\rho(t,r))$. In~\cite{AC18} it was shown that a similar equation of state can be imposed for spherically symmetric elastic bodies. 
\begin{definition}\label{elasticmaterialmodeldef}
Let $\kappa>0$ and $-1<\nu\leq 1/2$ be given. An elastic constitutive function for spherically symmetric bodies with bulk modulus $\kappa$ and Poisson ratio $\nu$ is a $C^2$ map $(\prad,\ptan):(0,\infty)^2\to\R^2$, such that $\prad(\delta,\eta)$, $\ptan(\delta,\eta)$ satisfy 
\begin{equation}\label{regularcentercond}
\widehat{p}_\mathrm{rad}(\delta,\delta)=\widehat{p}_\mathrm{tan}(\delta,\delta),\quad \partial_\eta(\prad+2\ptan)(\delta,\delta)=0,\quad\text{for all $\delta>0$,}
\end{equation}
as well as
\begin{subequations}\label{condpress}
\begin{equation}\label{naturalstate}
\prad(1,1)=\ptan(1,1)=0,
\end{equation}
\begin{equation}\label{lincomp1}
\kappa^{-1}\partial_\delta\prad(1,1)=\frac{3(1-\nu)}{1+\nu},\quad \kappa^{-1}\partial_\eta\prad(1,1)=-\frac{2(1-2\nu)}{1+\nu},
\end{equation}
\begin{equation}\label{lincomp2}
\kappa^{-1}\partial_\delta\ptan(1,1)=\frac{3\nu}{1+\nu},\quad\kappa^{-1}\partial_\eta\ptan(1,1)=\frac{1-2\nu}{1+\nu}.
\end{equation}
\end{subequations}
If there exists a $C^3$ function $\widehat{w}:(0,\infty)^2\to\R$ such that 
\begin{equation}\label{normalcond}
\widehat{w}(1,1)=0
\end{equation}
and
\begin{equation}\label{hyperelasticconst}
\widehat{p}_\mathrm{rad}(\delta,\eta)=\delta^2\partial_\delta\widehat{w}(\delta,\eta),\quad \widehat{p}_\mathrm{tan}(\delta,\eta)=\widehat{p}_\mathrm{rad}(\delta,\eta)+\frac{3}{2}\delta\eta\partial_\eta \widehat{w}(\delta,\eta),
\end{equation}
then the constitutive function is said to be hyperelastic with stored energy function $\widehat{w}$.
\end{definition}

\begin{definition}\label{elasticball}
A ball of matter $(\rho,p_\mathrm{rad},p_\mathrm{tan},u)$ is said to be an (homogeneous) elastic ball with reference density $\mathcal{K}>0$ and reference pressure $\mathcal{P}\geq0$ if there exists
an elastic constitutive function $(\prad,\ptan)$ for spherically symmetric bodies
such that the following equation of state holds:
\begin{subequations}\label{EOS}
\begin{align}
&\rho(t,r)=\mathcal{K}\delta(t,r),\\
&p_\mathrm{rad}(t,r)=\mathcal{P}+\prad(\delta(t,r),\eta(t,r)),\\
&p_\mathrm{tan}(t,r)=\mathcal{P}+\ptan(\delta(t,r),\eta(t,r)),
\end{align}
where
\begin{equation}\label{defeta}
\eta(t,r)=\frac{m(t,r)}{\frac{4\pi}{3}\mathcal{K}r^3},\quad r\in (0,R(t)).
\end{equation}
\end{subequations}
\end{definition}

{\it Remark.} The $(\delta,\eta)$-space in Definition~\ref{elasticmaterialmodeldef} corresponds to the Eulerian (or physical) space of configurations of the spherically symmetric elastic body. The point $(\delta,\eta)=(1,1)$  is called reference, or unstrained, configuration; in terms of the principal stretches $\lambda_1,\lambda_2$ we have $(\delta,\eta)=(1,1)\Leftrightarrow (\lambda_1,\lambda_2)=(1,1)$, see Section~\ref{lagsec}. The reference configuration is used in Definition~\ref{elasticball} to define the reference state of elastic balls.
In particular, according to Definition~\ref{elasticball}, (homogeneous) elastic balls have constant mass density $\mathcal{K}$ and constant (isotropic) pressure $\mathcal{P}$
in the reference state. When $\mathcal{P}=0$, as it is assumed in~\cite{AC18}, the reference state is said to be a natural state  

{\it Remark.} As a consequence of~\eqref{condpress}, all constitutive functions $(\prad,\ptan)$ for a given material (i.e., for given parameters $\kappa,\nu$) have the same linear approximation around the reference configuration $(\delta,\eta)=(1,1)$. This property expresses  the postulate of compatibility with linear elasticity in our formulation of (non-linear) elasticity theory for spherically symmetric bodies.



{\it Remark.} The Eulerian characterization of elastic balls given in Definition~\ref{elasticball} has an equivalent representation in the Lagrangian space, i.e., in the space of spatial deformations of the ball from the reference configuration, which is discussed in Section~\ref{lagsec}. In the Lagrangian formulation the radius of the ball is constant and given by
\[
Z=\left(\frac{3M}{4\pi\mathcal{K}}\right)^{1/3},
\]
where $M$ is the total mass of the ball, which is preserved by deformations. 
%

%

{\it Remark.} It is important to keep in mind the difference between the elastic constitutive function $(\prad,\ptan)$ and the equation of state~\eqref{EOS}. The former depends only on material properties, such as $\kappa$ and $\nu$, while the latter depends also on the reference state parameters $\mathcal{K},\mathcal{P}$ as well as, in general, on geometrical properties of the body. For instance, the equation of state for elastic shells is given by~\eqref{EOS} with the following alternate expression for the function $\eta$:
\[
\eta(t,r)=\left(\frac{\mathcal{S}}{r}\right)^3+\frac{3}{r^3}\int_{R_\mathrm{in}(t)}^{r}\delta(t,s)s^2ds,\quad r\in(R_\mathrm{in}(t),R(t)),
\]
where $R_{\mathrm{in}}(t)\in(0,R(t))$ is the inner radius of the shell and $\mathcal{S}$ is a physically dimensional constant which corresponds to the inner radius of the shell in the reference state, 
see~\cite{AC18}. In the case of fluid bodies, the equations of state of balls and shells are the same, because the fluid constitutive function is independent of $\eta$, see Section~\ref{fluidsec} below.

{\it Remark.} Let the shear variable $y$ be defined as
\begin{equation}\label{shear}
y=\frac{\delta}{\eta};
\end{equation}
in terms of the principal stretches $\lambda_1,\lambda_2$, we have $y=\lambda_2/\lambda_1$, see Section~\ref{lagsec}.
The condition~\eqref{regularcentercond} signifies that the state $y=1$ is isotropic and ensures in particular that the center of static elastic balls is strongly regular. Moreover it can be shown that for hyperelastic constitutive functions the second condition in~\eqref{regularcentercond} follows by the first one, see~\cite[Thm.~1]{AC19}.
 


{\it Remark.} In~\cite{AC18,AC19} the Lam\'e parameters $\lambda,\mu$ are used instead of $\kappa,\nu$. The relation between these two pairs of material constants is
\[
\kappa=\lambda+\frac{2\mu}{3},\quad\nu=\frac{\lambda}{2(\lambda+\mu)}.
\]
In this paper the Poisson ratio plays an important role, which is the reason to deviate from the notation used in~\cite{AC18,AC19}.

{\it Remark.} Letting 
\begin{subequations}\label{equivhyper}
\begin{equation}
h_1(\delta,\eta)=\delta^{-2}\prad(\delta,\eta),\quad h_2(\delta,\eta)=\frac{2}{3}(\delta\eta)^{-1}(\ptan(\delta,\eta)-\prad(\delta,\eta)),
\end{equation}
an elastic constitutive equation is hyperelastic if and only if the form $h_1(\delta,\eta)d\delta+h_2(\delta,\eta)d\eta$ is exact. Hence we have the following equivalency: 
\begin{equation}
(\prad,\ptan) \ \text{hyperelastic} \Leftrightarrow \partial_\eta h_1=\partial_\delta h_2.
\end{equation}
\end{subequations}



\subsection{Fluid constitutive functions}\label{fluidsec}
In this section we discuss how the special case of (barotropic) fluid balls fits into our formulation of elasticity theory for spherically symmetric bodies.
\begin{definition}\label{fluiddef}
An elastic constitutive function $(\prad,\ptan)$ for spherically symmetric bodies is said to be a (barotropic) fluid constitutive function if there exists $\widehat{p}:(0,\infty)\to\R$ such that
\[
\prad(\delta,\eta)=\ptan(\delta,\eta)=\widehat{p}(\delta).
\]
A ball of matter $(\rho,p_\mathrm{rad},p_\mathrm{tan},u)$ is said to be a fluid ball with reference density $\mathcal{K}>0$ and reference pressure $\mathcal{P}\geq0$ if there exists a fluid constitutive function $\widehat{p}$ such that the following equation of state holds
\begin{equation}\label{EOSfluid}
\rho(t,r)=\mathcal{K}\delta(t,r),\quad p_\mathrm{rad}(t,r)=p_\mathrm{tan}(t,r)=p(t,r)=\mathcal{P}+\widehat{p}(\delta(t,r)).
\end{equation}
\end{definition}
{\it Remark.}
By~\eqref{condpress}, $\nu=1/2$ and $\widehat{p}\,'(1)=\kappa$ must hold for a fluid constitutive function. Moreover all fluid constitutive functions are hyperelastic.

{\it Remark.} Definition~\ref{fluiddef} is perfectly meaningful even for shells of matter, or even if the body is not spherically symmetric, but in the following it will be applied only to fluid balls.
%

The important example of polytropic fluid balls, i.e., fluid balls with equation of state~\eqref{polyfluid}, is obtained by choosing the constitutive function and the reference pressure as
\begin{equation}\label{constpolyfluid}
\widehat{p}(\delta)=\frac{\kappa}{\gamma}\,(\delta^{\gamma}-1),\quad \mathcal{P}=\frac{\kappa}{\gamma}=-\widehat{p}(0),
\end{equation} 
which we call respectively the polytropic fluid constitutive function and the polytropic reference pressure. The constant $c>0$ in~\eqref{polyfluid} is given by $c=\kappa/(\gamma\mathcal{K}^\gamma)$; the stored energy function of polytropic fluid balls is
\begin{equation}\label{polystored}
\widehat{w}_\mathrm{pf}(\delta,\eta)=\widehat{w}_\mathrm{pf}(\delta)=\frac{\kappa}{\gamma}\left(\frac{\delta^{\gamma-1}-1}{\gamma-1}+\delta^{-1}-1\right).
\end{equation} 
The advantage of this re-formulation of the polytropic fluid model is that it arises from a unified description of fluid and elastic bodies. 

%

\subsection{Dynamically equivalent constitutive functions}\label{dynsec}
The system~\eqref{sssystem} for elastic balls becomes
\begin{subequations}\label{sssystem2}
\begin{align}
&\partial_t\delta+\frac{1}{r^2}\partial_r(r^2\delta u)=0,\label{continuityeq}\\
&\mathcal{K}\delta(\partial_t u+u\partial_r u)=-\widehat{a}(\delta,\eta)\partial_r\delta+\widehat{b}(\delta,\eta)\frac{\eta-\delta}{r}+\mathcal{K}\delta f\label{equ},
\end{align}
where 
\begin{equation}\label{abc}
\widehat{a}(\delta,\eta)=\partial_\delta\prad(\delta,\eta),\quad \widehat{b}(\delta,\eta)=2\frac{\widehat{p}_\mathrm{tan}(\delta,\eta)-\widehat{p}_\mathrm{rad}(\delta,\eta)}{\eta-\delta}+3\partial_\eta\widehat{p}_\mathrm{rad}(\delta,\eta).
\end{equation}
Moreover $\eta(t,r)$ satisfies the equations
\begin{equation}\label{eqseta}
\partial_t\eta+u\partial_r\eta=-3\frac{u}{r}\eta,\quad \partial_r\eta=-\frac{3}{r}(\eta-\delta),\quad \partial_t\eta=-3\frac{u}{r}\delta.
\end{equation}
\end{subequations}
The quantity $\widehat{b}(\delta,\eta)$ in~\eqref{equ} is the new term that distinguishes the equations of motion of elastic balls from those of fluid balls. As
$\lim_{\delta\to\eta}\widehat{b}(\delta,\eta)=\partial_\eta(\prad+2\ptan)(\delta,\delta)$, then, by~\eqref{regularcentercond}, 
\begin{equation}\label{bzero}
\lim_{\delta\to\eta}\widehat{b}(\delta,\eta)=0,
\end{equation}
i.e.,  the function $\widehat{b}$ vanishes on the isotropic state $\delta=\eta$.


The system~\eqref{sssystem2} depends only on the elastic constitutive function and the reference density $\mathcal{K}$; it is independent of the reference pressure $\mathcal{P}$. However different elastic constitutive functions give rise to the same system~\eqref{sssystem2} on the variables $(\delta,\eta,u)$ when the conditions in the following definition are satisfied.
\begin{definition}
Two constitutive functions $(\prad,\ptan)$, $(\ptilderad,\ptildetan)$ are said to be dynamically equivalent if $\widehat{a}(\delta,\eta)=\widetilde{a}(\delta,\eta)$ and $\widehat{b}(\delta,\eta)=\widetilde{b}(\delta,\eta)$, for all $\delta,\eta>0$.
\end{definition}
In the fluid case, dynamically equivalent constitutive functions are necessarily identical, i.e., $\widehat{p}(\delta)=\widetilde{p}(\delta)$. In the more general case of elastic materials, we have the following simple result. 
\begin{lemma}\label{dyneqlemma}
Two elastic constitutive functions $(\prad,\ptan)$, $(\ptilderad,\ptildetan)$ for spherically symmetric bodies are dynamically equivalent if and only if there exists a $C^2$ function $q:(0,\infty)\to\R$ such that $q(1)=q'(1)=0$ and 
\begin{equation}\label{dynequiv}
\prad(\delta,\eta)=\ptilderad(\delta,\eta)+q(\eta),\quad \ptan(\delta,\eta)=\ptildetan(\delta,\eta)+q(\eta)-\frac{3}{2}(\eta-\delta)q'(\eta).
\end{equation}
\end{lemma}
\begin{proof}
The condition $\widehat{a}(\delta,\eta)=\widetilde{a}(\delta,\eta)$
is equivalent to $\prad(\delta,\eta)=\ptilderad(\delta,\eta)+q(\eta)$, 
where, by~\eqref{condpress}, $q(1)=q'(1)=0$. Substituting in the condition $\widehat{b}(\delta,\eta)=\widetilde{b}(\delta,\eta)$ gives the second equation in~\eqref{dynequiv}.\end{proof}
\begin{lemma}\label{bfluid}
If the constitutive function $(\prad,\ptan)$ is hyperelastic and $\widehat{b}(\delta,\eta)\equiv0$, then $(\prad,\ptan)$ is dynamically equivalent to a fluid constitutive function.
\end{lemma}
\begin{proof}
We prove the result by showing that, under the stated assumptions, $\partial_\delta\prad$ is independent of $\eta$. By~\eqref{hyperelasticconst} we have
\begin{align*}
\textfrac{1}{3}\partial_\delta \widehat{b}(\delta,\eta)&=\partial_\eta\partial_\delta\prad(\delta,\eta)+\partial_\delta\left(\frac{\delta\eta}{\eta-\delta}\partial_\eta\widehat{w}(\delta,\eta)\right)\\
&=\partial_\eta\partial_\delta\prad(\delta,\eta)+\frac{\eta^2}{(\eta-\delta)^2}\partial_\eta\widehat{w}(\delta,\eta)+\frac{\delta\eta}{\eta-\delta}\partial_\delta\partial_\eta\widehat{w}(\delta,\eta)\\
&=\partial_\eta\partial_\delta\prad(\delta,\eta)+\frac{2}{3}\frac{\eta}{\delta(\eta-\delta)^2}(\ptan(\delta,\eta)-\prad(\delta,\eta))+\frac{\eta}{\delta(\eta-\delta)}\partial_\eta\prad(\delta,\eta)\\
&=\partial_\eta\partial_\delta\prad(\delta,\eta)+\frac{\eta}{3\delta(\eta-\delta)}\partial_\delta \widehat{b}(\delta,\eta),
\end{align*}
by which the claim follows.
\end{proof} 

\subsection{Lagrangian formulation}\label{lagsec}
The purpose of this section is to derive the Lagrangian formulation for the evolution problem of elastic balls implied by the Eulerian formulation in physical space described so far. The relation between the two formulations in the static case is discussed in~\cite{AC18}. For simplicity we set $f\equiv 0$ in this section.

Let $(\delta,\eta,u)$ be a smooth solution of the system~\eqref{sssystem2}$_{f\equiv 0}$ describing the motion of an elastic ball with radius $R(t)$. We assume in addition that $u$ satisfies the boundary condition $u(t,R(t))=\dot{R}(t)$, so that in particular the total mass $M$ of the ball is conserved.
Define 
\begin{subequations}\label{configurationmap}
\begin{equation}\label{Q}
Z=\left(\frac{3M}{4\pi\mathcal{K}}\right)^{1/3},\quad \phi_t: [0,R(t)]\to [0,Z],\quad \phi_t(r)=(\eta(t,r)r^3)^{1/3}.
\end{equation}
$\phi_t$ is monotonically increasing and satisfies $\phi_t(0)=0$, $\phi_t(R(t))=Z$, $\phi'_t(0)=1$.
Hence 
\begin{equation}\label{psi}
\psi_t=\phi_t^{-1}:[0,Z]\to[0,R(t)]
\end{equation}
is also monotonically increasing and satisfies $\psi_t(0)=0$, 
$\psi_t(Z)=R(t)$, $\psi'_t(0)=1$. Letting 
\begin{equation}
\mathcal{B}=\{X\in\R^3:|X|=z\in [0,Z)\},
\end{equation}
the deformation function is defined as 
\begin{equation}
\Psi_t:\overline{\mathcal{B}}\to\R^3,\quad \Psi_t(X)=\frac{\psi_t(z)}{z}X.
\end{equation}
The region $\mathcal{B}$ is the interior of the elastic ball in the reference configuration, while 
\[
\Psi_t(\mathcal{B})=\{x\in\R^3: |x|=r\in [0,R(t))\}
\] 
is the interior of the elastic ball in physical space.  $\Psi_t$ is a diffeomorphism from $\overline{\mathcal{B}}$ onto $\overline{\Psi_t(\mathcal{B})}$. The deformation gradient is 
\[
F_{ij}(X)=(\nabla\Psi_t)_{ij}(X)=\psi_t'(z)\frac{X_iX_j}{z^2}+\frac{\psi_t(z)}{z}\left(\delta_{ij}-\frac{X_iX_j}{z^2}\right),\quad X\in\mathcal{B},\ z=|X|.
\]
Letting
\[
\det(F(X))=\psi'_t(z)\left(\frac{\psi_t(z)}{z}\right)^2=J_t(z),
\]
and using the identity  $\phi'_t(r)=\delta(t,r)/\eta(t,r)^{2/3}$ we find
$\delta(t,r)=J_t(\phi_t(r))^{-1}$.
Hence the last equation in~\eqref{eqseta} gives
\[
u(t,r)=-\frac{r\partial_t\eta(t,r)}{3\delta(t,r)}=-\left(\frac{\phi_t(r)}{r}\right)^2J_t(\phi_t^{-1}(r))\partial_t\phi_t(r)=-\psi'_t(\phi_t(r))\partial_t\phi_t(r)=\partial_t\psi_t(\phi_t(r)),
\]
where in the last step we used $\psi_t(\phi_t(r))=r$. It follows that 
\begin{equation}\label{dtpsi}
\partial_t\Psi_t:\overline{\mathcal{B}}\to\R^3,\quad \partial_t\Psi_t=\frac{u_t\circ\psi_t^{-1}(z)}{z}X.
\end{equation}
\end{subequations}
Equations~\eqref{configurationmap} define the deformation function and its time derivative in terms of the solution $(\delta,\eta,u)$ of the system~\eqref{sssystem2}. Conversely, given a deformation $\Psi_t:\overline{\mathcal{B}}\to\R^3$ of the form $\Psi_t(X)=\psi_t(z)X/z$, where $\psi_t:[0,Z]\to\R$ is $C^1$ and monotonically increasing, the functions 
%
%
\begin{equation}\label{functions}
u(t,r)=\partial_t\psi_t(z),\quad \delta(t,r)=\frac{z^2}{\psi_t(z)^2\psi_t'(z)}, \quad \eta(t,r)=\left(\frac{z}{\psi_t(z)}\right)^3,\quad z=\psi^{-1}_t(r)
\end{equation}
satisfy~\eqref{continuityeq} and~\eqref{eqseta} for $r\in (0,R(t))$, where $R(t)=\psi_t(Z)$.
If in addition $\psi(t,z)=\psi_t(z)$ solves the  non-linear wave equation 
\begin{subequations}\label{NLW}
\begin{equation}
\mathcal{K}\partial_t^2\psi-A(z,\psi,\partial_z\psi)\partial_z^2\psi+\frac{B(z,\psi,\partial_z\psi)}{\psi}=0,\quad t>0, \ z\in (0,Z),
\end{equation}
where 
\begin{align}
&A(z,\psi,\partial_z\psi)=\frac{1}{(\partial_z\psi)^2}\Big\{\widehat{a}(\delta,\eta)\Big\}_{\displaystyle{(\delta=z^2/(\psi^2\partial_z\psi),\eta=(z/\psi)^3})}\\[0.5cm]
&B(z,\psi,\partial_z\psi)=\Big\{(\delta-\eta)\Big(\frac{2\,\widehat{a}(\delta,\eta)}{\eta}+\frac{\widehat{b}(\delta,\eta)}{\delta}\Big)\Big\}_{\displaystyle{(\delta=z^2/(\psi^2\partial_z\psi),\eta=(z/\psi)^3})}
\end{align}
\end{subequations}
then $(\delta,\eta,u)$ given by~\eqref{functions} satisfies the balance of momentum equation~\eqref{equ}$_{f\equiv 0}$ as well. 
The boundary condition $u(t,R(t))=\dot{R}(t)$ is identically satisfied, since $R(t)=\psi(t,Z)$ and $u(t,R(t))=\partial_t\psi(t,Z)$. 

{\it Remark.} While the transition $(\delta,\eta,u)\to (\Psi,\partial_t\Psi)$ is possible for all sufficiently smooth solutions of~\eqref{sssystem2}, the opposite transition $(\Psi,\partial_t\Psi)\to(\delta,\eta,u)$ is only possible if the configuration map solution of~\eqref{NLW} satisfies $\partial_z\psi(t,z)>0$, for all $z\in [0,Z]$.

\subsubsection*{Stored energy function}
Let $(\Lambda_1^2,\Lambda_2^2,\Lambda_3^2)$ be the eigenvalues of the 
(right) Cauchy-Green tensor $C=F^TF$.
In the Lagrangian formulation of elasticity, the stored energy function of homogeneous, isotropic and frame indifferent hyperelastic materials depends only the principal stretches  $(\Lambda_1,\Lambda_2,\Lambda_3)$.  In the case of spherically symmetric deformations we have
\[
C_{ij}(X)=(\psi_t'(z))^2\frac{X_iX_j}{z^2}+\left(\frac{\psi_t(z)}{z}\right)^2\left(\delta_{ij}-\frac{X_iX_j}{z^2}\right),\quad X\in\mathcal{B},\ z=|X|,
\]
hence $\Lambda_1(t,z)=\psi'_t(z), \Lambda_2(t,z)=\Lambda_3(t,z)=\psi_t(z)/z$, or equivalently
\[
\Lambda_i(t,z)=\lambda_i(t,\psi_t(z)),\quad \lambda_1(t,r)=\frac{\eta(t,r)^{2/3}}{\delta(t,r)}\quad \lambda_2(t,r)=\lambda_3(t,r)=\frac{1}{\eta(t,r)^{1/3}},
\]
from which we derive in particular $y=\delta/\eta=\lambda_2/\lambda_1$. Letting $\widehat{W}(\Lambda_1,\Lambda_2)$ be a stored energy function for spherically symmetric deformations in the Lagrangian space, the expression
\[
\widehat{w}(\delta,\eta)=\widehat{W}(\delta^{-1}\eta^{2/3},\eta^{-1/3}),
\]
defines the stored energy function in the Eulerian space used in Definition~\ref{elasticmaterialmodeldef}.
For instance the stored energy function of  Saint-Venant Kirchhoff (SVK) materials for spherically symmetric deformations is given by
\[
\kappa^{-1}\widehat{W}_{SVK}(\Lambda_1,\Lambda_2)=-\frac{3}{4}(\Lambda_1^2+2\Lambda_2^2)+\frac38\frac{1-\nu}{1+\nu}(\Lambda_1^4+2\Lambda_2^4)+\frac34\frac{\nu}{1+\nu}(\Lambda_2^4+2\Lambda_1^2\Lambda_2^2)+\frac{9}{8},
\]
see~\cite[p.~183]{Ciarlet}.
Hence in physical space we find
\begin{align}\label{SVK}
\kappa^{-1}\widehat{w}_\mathrm{SVK}(\delta,\eta)&=\eta^{-4/3}\left(\frac{3(1-\nu)}{8(1+\nu)}\left(\frac{\delta}{\eta}\right)^{-4}+\frac{3\nu}{2(1+\nu)}\left(\frac{\delta}{\eta}\right)^{-2}+\frac{3}{4(1+\nu)}\right)\nonumber\\
&\quad+\eta^{-2/3}\left(-\frac{3}{4}\left(\frac{\delta}{\eta}\right)^{-2}-\frac{3}{2}\right)+\frac{9}{8}.
\end{align}

The constitutive function for the principal pressures in the Lagrangian space is given in terms of the stored energy function by
\begin{subequations}\label{servonodopo}
\begin{equation}
\widehat{P}_\mathrm{rad}(\Lambda_1,\Lambda_2)=-\partial_1\widehat{W}(\Lambda_1,\Lambda_2),\quad\ \widehat{P}_\mathrm{tan}(\Lambda_1,\Lambda_2)=-\frac{1}{2}\partial_2\widehat{W}(\Lambda_1,\Lambda_2),
\end{equation}
see~\cite{AC18},
which are equivalent to the formulas~\eqref{hyperelasticconst} in the Eulerian space, with
\begin{equation}
\prad(\delta,\eta)=\eta^{2/3}\widehat{P}_\mathrm{rad}(\delta^{-1}\eta^{2/3},\eta^{-1/3}),\quad\ptan(\delta,\eta)=\frac{\delta}{\eta^{1/3}}\widehat{P}_\mathrm{tan}(\delta^{-1}\eta^{2/3},\eta^{-1/3}).
\end{equation}
\end{subequations}
%
%

\subsection{Constitutive inequalities}\label{constsec}
The constitutive function $(\prad,\ptan)$ will now be required to satisfy three important inequalities in continuum mechanics.
The first one follows by demanding that the system~\eqref{sssystem2} with $f\equiv 0$ be strictly hyperbolic. Letting $q=\delta u$ be the local (radial) momentum and
\[
 v=\begin{pmatrix} \delta \\ \eta \\ q\end{pmatrix},\quad A(v)=\begin{pmatrix}0 & 0 & 1\\ 0 & \frac{q}{\delta} & 0 \\ -(\frac{q}{\delta})^2+\frac{\widehat{a}(\delta,\eta)}{\mathcal{K}}& \frac{\widehat{b}(\delta,\eta)}{3\mathcal{K}} & \frac{2q}{\delta}\end{pmatrix},\quad F(r,v)=-\frac{1}{r\delta}\begin{pmatrix}2q\delta \\ {\displaystyle  3q\eta} \\ 2q^2 \end{pmatrix},
\]
the system~\eqref{sssystem2}$_{f\equiv 0}$ can be written in the form 
\[
\partial_t v+A(v)\partial_r v=F(r,v);
\]
the function $F(r,v)$ is bounded at $r=0$ for solutions with strongly regular center. 
The eigenvalues of $A$ are 
\[
\mathrm{Eigen}(A)=(\frac{q}{\delta},\frac{q}{\delta}-\sqrt{\widehat{a}(\delta,\eta)},\frac{q}{\delta}+\sqrt{\widehat{a}(\delta,\eta)})=(u,u-\sqrt{\partial_\delta\prad(\delta,\eta)},u+\sqrt{\partial_\delta\prad(\delta,\eta)}).
\]
Hence the system~\eqref{sssystem2}$_{f\equiv 0}$ is strictly hyperbolic if $\partial_\delta\prad(\delta,\eta)>0$. 
\begin{definition}\label{hyperdef}
An elastic constitutive function $(\prad,\ptan)$ for spherically symmetric bodies is said to be globally strictly hyperbolic if
\begin{equation}\label{hypercond}
\partial_\delta\prad(\delta,\eta)>0,\quad \text{for all $(\delta,\eta)\in (0,\infty)^2$.}
\end{equation}
If the constitutive function is hyperelastic then the stored energy function itself is said to be globally strictly hyperbolic. 
\end{definition}
A further condition that is expected to be satisfied by all reasonable elastic constitutive functions
is the Baker-Ericksen inequality~\cite{BEref}.  In the Lagrangian formulation of elasticity theory, the Baker-Ericksen inequality is an assumption on the stored energy function implying that the deformation in the interior of the body along each principal direction increases with the intensity of the applied force, see~\cite{MH}. More precisely, 
let $\psi:\mathcal{B}\to\R^3$ be the deformation function of a body from a reference configuration in which it occupies the region $\mathcal{B}\subset\R^3$. Let $F=\nabla\psi$ be the deformation gradient, $C=F^TF$ the (right) Cauchy-Green tensor and $\Lambda_1^2,\Lambda_2^2,\Lambda_3^2$ the eigenvalues of $C$. Let $\Phi(\Lambda_1,\Lambda_2,\Lambda_3)$ be a stored energy function for homogeneous, isotropic and frame indifferent elastic materials. 
Then the Baker-Ericksen  inequality reads
\begin{equation}\label{BElag}
\frac{\Lambda_i\partial_i\Phi-\Lambda_j\partial_j\Phi}{\Lambda_i-\Lambda_j}>0,\quad\Lambda_i\neq\Lambda_j.
\end{equation}
As shown in~\cite{MH}, the Baker-Ericksen inequality is implied by  the strong ellipticity condition on the stored energy function in the Lagrangian space.

The form of the Baker-Ericksen inequality in our Eulerian formulation of elasticity theory for spherically symmetric bodies is derived using~\eqref{servonodopo} with $\widehat{W}(\Lambda_1,\Lambda_2)= \Phi(\Lambda_1,\Lambda_2,\Lambda_2)$ and reads $(\delta-\eta)^{-1}(\ptan(\delta,\eta)-\prad(\delta,\eta))>0$, for all $\eta\neq \delta$. However for our purpose it is convenient to slightly modify this inequality to the following. 
\begin{definition}\label{BEdef}
The inequality
\begin{equation}\label{BEineq}
\frac{\ptan(\delta,\eta)-\prad(\delta,\eta)}{\eta-\delta}\geq0,
\end{equation}
on the elastic constitutive function $(\prad,\ptan)$ for spherically symmetric bodies is called strong Baker-Ericksen inequality if it is satisfied for all $(\delta,\eta)\in (0,\infty)^2$, and weak Baker-Ericksen inequality if it is satisfied for all $\eta\geq\delta>0$.
\end{definition}

The inequality in Definition~\ref{BEdef} differs from~\eqref{BElag} in three ways. Firstly, we introduced a weak and a strong version of the Baker-Ericksen inequality;  the reason for this is that the weak version suffices in the important case of static bodies, see Lemma~\ref{propostatic}. Secondly we replaced the strict inequality in~\eqref{BElag} with the non-strict one, so that~\eqref{BEineq} can also be applied to fluid modes. Finally the condition $\Lambda_i\neq \Lambda_j$, which in our formulation is equivalent to $\eta\neq\delta$, has not been imposed in~\eqref{BEineq}. Hence we assume that the Baker-Ericksen inequality still holds in the limit $\eta\to\delta$.

To justify the next constitutive inequality, we remark that, by~\eqref{condpress}, $\nabla\widehat{w}(1,1)=0$ and
\[
\kappa^{-1}\nabla^2\widehat{w}(1,1)=\frac{1}{1+\nu}\begin{pmatrix}3(1-\nu)&-2(1-2\nu)\\-2(1-2\nu)& 2(1-2\nu)\end{pmatrix},
\]
which is positive definite for $\nu\in(-1,1/2)$. 
It follows that the stored energy function has a local minimum at the reference configuration $(\delta,\eta)=(1,1)$, which we shall now require to be global (but not necessarily unique).
As for the Baker-Ericksen inequality, we introduce a weak and a strong version of this condition, the former being sufficient in the case of static solutions.

\begin{definition}\label{posenergy}
An hyperelastic constitutive function with stored energy function $\widehat{w}$ is said to satisfy the weak, respectively strong, non-negative stored energy condition if 
\begin{equation}
\widehat{w}(\delta,\eta)\geq 0,\quad\text{for all $\eta\geq\delta>0$, respectively for all $(\delta,\eta)\in (0,\infty)^2$.}
\end{equation}
\end{definition}

{\it Remark.}
Definitions~\ref{hyperdef}--\ref{posenergy} do not exclude the possibility that the constitutive inequalities in this section only hold for some, and not all, values of the Poisson ratio $\nu$, i.e., only for specific materials.

{\it Remark.} The polytropic fluid constitutive function~\eqref{constpolyfluid} is globally strictly hyperbolic and satisfies the strong non-negative stored energy condition as well as the strong Baker-Ericksen inequality.

It turns out that several elastic models used in materials science satisfy the non-negative stored energy condition, but are neither globally hyperbolic nor satisfy the Baker-Ericksen inequality. 
For instance, the region in the $(\delta,\eta)$ plane where the SVK model~\eqref{SVK} is hyperbolic and verifies the Baker-Ericksen inequality is depicted in Figure~\ref{figura}. In particular, these two important constitutive inequalities are violated for large deformations. 

\begin{figure}[ht!]
\begin{center}
\subfigure{
\includegraphics[scale=0.6]{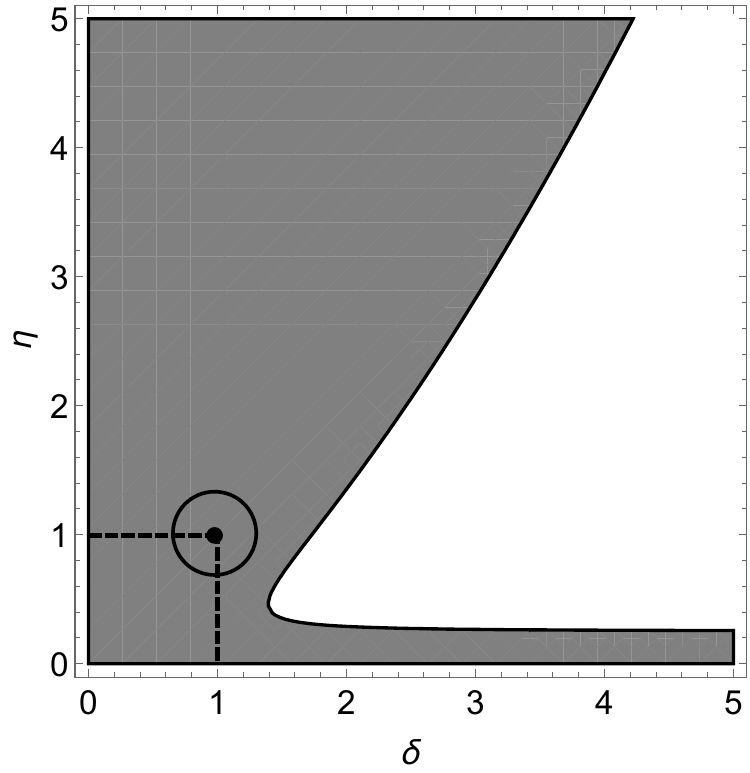}}\qquad
\subfigure{
\includegraphics[scale=0.6]{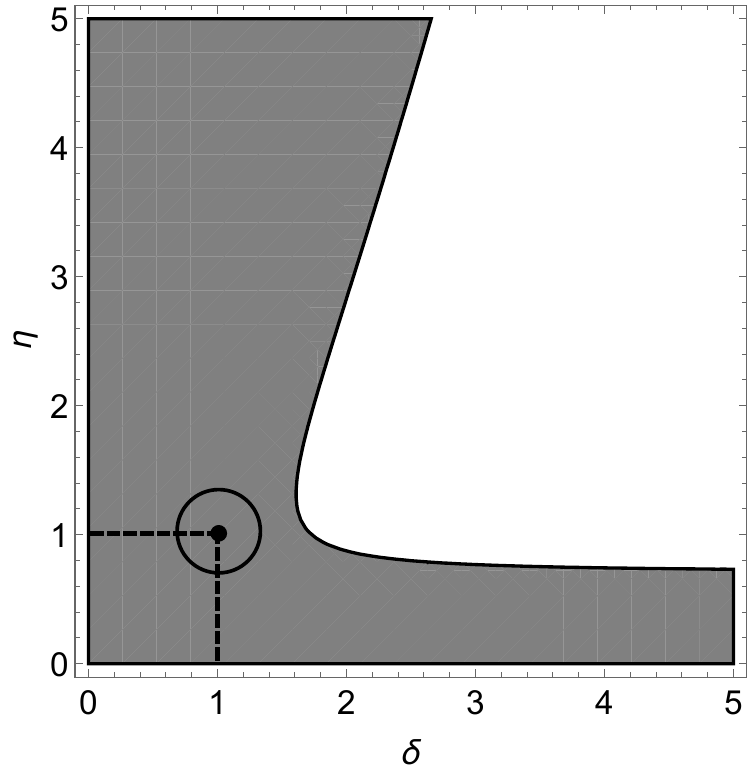}}
\caption{The region in the $(\delta,\eta)$ plan where the hyperbolicity condition (left) and the Baker-Ericksen inequality (right) are satisfied by the SVK model~\eqref{SVK}. Both are verified in a neighbourhood of the reference state $(\delta,\eta)=(1,1)$. The Poisson ratio is $\nu=1/4$ in this picture. }
\label{figura}
\end{center}
\end{figure}

To conclude this section we show that the Baker-Ericksen inequality entails a simple condition for the existence of a global minimum of the stored energy function.

\begin{lemma}\label{wposlemma}
Let $(\prad,\ptan)$ be an hyperelastic constitutive function for spherically symmetric bodies satisfying the strong, respectively weak, Baker-Ericksen inequality. 
Then the stored energy function satisfies
\[
\widehat{w}(\delta,\eta)\geq\widehat{w}(\delta,\delta),\quad \text{for all $(\delta,\eta)\in (0,\infty)^2$, respectively for all $\eta\geq\delta>0$.}
\]
\end{lemma}
\begin{proof}
The proof is straightforward, because the inequality~\eqref{BEineq} applied to hyperelastic materials is equivalent to the property that, for all given $\delta>0$, the function $\eta\to\widehat{w}(\delta,\eta)$ is non-increasing for $\eta<\delta$ and non-decreasing for $\eta>\delta$.
\end{proof}

Lemma~\ref{wposlemma} reduces the problem of proving the non-negativity of the stored energy function $\widehat{w}$ to the much simpler problem of proving the same property for the function $F$ of one variable given by $F(\delta)=\widehat{w}(\delta,\delta)$.

\subsection{Constant boundary shear condition}\label{remhomo}

In this paper we are particularly interested in (self-gravitating) elastic balls with constant boundary shear, i.e., elastic balls for which there exists a constant $y_\mathrm{b}\geq0$ such that
\begin{equation}\label{CBScond}
y(t,R(t))=y_\mathrm{b}
\end{equation} 
for all times. As we shall see, the constant boundary shear condition generalizes to the elastic case the property that the mass density vanishes on the boundary of polytropic fluid balls (surrounded by vacuum). We show now that~\eqref{CBScond} is also necessary for the existence of homologous motions for elastic balls.
A ball of matter $(\rho,p_\mathrm{rad},p_\mathrm{tan},u)$ is said to be in homologous motion
if the velocity field $u(t,r)$ has the form 
\begin{equation}\label{ansatzu}
u(t,r)=\frac{\dot{\omega}(t)}{\omega(t)}r,\quad t\geq 0,\ 0\leq r\leq R(t),
\end{equation}
for some $C^1$ function $\omega$ such that $\omega(0)=1$. 
Replacing~\eqref{ansatzu} into~\eqref{continuityeq} we obtain 
\[
\delta(t,r)=\frac{1}{\omega(t)^3}\delta_0\left(\frac{r}{\omega(t)}\right),\quad \eta(t,r)=\frac{1}{\omega(t)^3}\eta_0\left(\frac{r}{\omega(t)}\right),\quad \eta_0(r)=\frac{3}{r^3}\int_0^r\delta_0(s)s^2\,ds,
\]
where $\delta_0=\delta(0,r)$. 
The boundary condition $\dot{R}(t)=u(t,R(t))$ defines the radius of the ball as $R(t)=Z\omega(t)$, where $Z>0$ is a constant. Hence 
\[
y(t,R(t))=\frac{\delta(t,R(t))}{\eta(t,R(t))}=\frac{\delta_0(Z)}{\eta_0(Z)}=y_\mathrm{b},
\]
and thus, as claimed, the boundary of elastic balls in homologous motion must have constant shear. 

The constitutive function for the radial pressure must be consistent with~\eqref{CBScond} in order for elastic balls with constant boundary shear to satisfy the boundary condition $p_\mathrm{rad}(t,R(t))=0$. As
\[
p_\mathrm{rad}(t,R(t))=\mathcal{P}+\prad(y_\mathrm{b}\eta(t,R(t)),\eta(t,R(t))), \quad\eta(t,R(t))=\frac{3M}{4\pi\mathcal{K}}R(t)^{-3},
\]
where $M$ is the total mass of the ball, a necessary condition on the constitutive function $(\prad,\ptan)$ for the existence of elastic balls with constant boundary shear is the following.
\begin{definition}\label{CBSdef}
An elastic constitutive function $(\prad,\ptan)$ for spherically symmetric bodies is said to satisfy the Constant Boundary Shear (CBS) condition if there exists a (possibly not unique) $y_b\geq0$ such that $\prad(y_\mathrm{b}\eta,\eta)$ is constant, or equivalently,
\begin{equation}\label{CBSeq}
(y_\mathrm{b}\partial_\delta\prad+\partial_\eta\ptan)(y_\mathrm{b}\eta,\eta)=0,\quad\text{for all $\eta>0$.}
\end{equation}
\end{definition}
The CBS condition is only necessary and by no means sufficient for the existence of elastic balls with constant boundary shear. An additional necessary condition is that the reference pressure in the equation of state~\eqref{EOS} of the elastic ball is given by
\begin{equation}\label{eqP}
\mathcal{P}=-\prad(y_\mathrm{b},1).
\end{equation}
Globally strictly hyperbolic fluid constitutive functions $\widehat{p}$ satisfy the CBS condition with $y_\mathrm{b}=0$ and thus in this case the reference pressure~\eqref{eqP} is $\mathcal{P}=-\widehat{p}(0)$. For polytropic fluid balls this leads to the choice $\mathcal{P}=\kappa/\gamma$, which characterizes the equation of state of polytropic fluid balls, see~\eqref{constpolyfluid}. 


{\it Remark.} By~\eqref{functions}, the deformation function on the boundary of an elastic ball with constant boundary shear satisfies 
\begin{equation}\label{BC}
\psi(t,Z)-y_\mathrm{b}Z\partial_z\psi(t,Z)=0,
\end{equation}
where $Z$ is the constant radius of the ball in the Lagrangian space, see Section~\ref{lagsec}.

\subsection{Remarks on static self-gravitating elastic balls}\label{remstatic}
This section contains some general remarks on static self-gravitating elastic balls with strongly regular center. 
A spherically symmetric matter distribution is said to be in static equilibrium if $u=0$ and $(\rho,p_\mathrm{rad},p_\mathrm{tan})$ are time independent. For self-gravitating static balls the system~\eqref{sssystem} becomes
\begin{equation}\label{generalstatic}
p_\mathrm{rad}'=\frac{2}{r}(p_\mathrm{tan}-p_\mathrm{rad})-G\rho\frac{m}{r^2}.
\end{equation} 
When the static ball is elastic,~\eqref{generalstatic} is equivalent to~\eqref{sssystem2} with $u=0$ and $f=-\frac{4\pi G}{3}\mathcal{K}r\eta$, namely
\begin{equation}\label{staticsystem}
\widehat{a}(\delta,\eta)\delta'=\frac{\widehat{b}(\delta,\eta)}{r}(\eta-\delta)- \frac{4\pi G}{3}\mathcal{K}^2r\,\eta\,\delta,\quad \eta(r)=\frac{3}{r^3}\int_0^r\delta(s)s^2\,ds.
\end{equation}

A solution $\delta$ of~\eqref{staticsystem} in the interval $[0,R)$ will be called regular if  $\delta\in C^0([0,R))\cap C^1((0,R))$ and $\delta(r)>0$, for all $r\in [0,R)$; the identity $\eta(0)=\delta(0)$ holds for these solutions. If in addition $\delta\in C^1([0,R))$ and $\lim_{r\to 0^+}\delta'(r)=0$, then $\delta$ is said to be a strongly regular solution. 

\begin{lemma}\label{propostatic}
Let $\delta$ be a a regular solution of~\eqref{staticsystem} defined on the interval $[0,R)$, $R\in(0,\infty]$. 
\begin{itemize}
\item[(i)] If $\widehat{a}(\delta(0),\delta(0))\neq 0$, then $\delta$ is strongly regular and $|\delta'(r)|\leq C r$, for a positive constant $C>0$ and $r<1$. Moreover $\delta\in C^2([0,R))$ and at $r=0^+$  there holds
\[
\delta''(0)=-\frac{4\pi G}{3}\mathcal{K}^2\frac{\delta(0)^2}{a(\delta(0),\delta(0))},\quad\eta''(0)=\frac{3}{5}\delta''(0).
\]
\item[(ii)] If $\widehat{a}(\delta,\delta)>0$, for all $\delta>0$, then $\eta(r)>\delta(r)$ and $\eta'(r)<0$, for all $r\in (0,R)$.
\item[(iii)] If $\widehat{a}(\delta,\delta)>0$, for all $\delta>0$, and the constitutive function satisfies the weak Baker-Ericksen inequality, then 
\begin{equation}\label{pradptan*}
\ptan(\delta(r),\eta(r))\geq\prad(\delta(r),\eta(r)),
\end{equation}
for all $r>0$. 
\end{itemize}
\end{lemma}
\begin{proof}
(i) The claim that $\delta$ is strongly regular and the bound on $\delta'$ are proved in~\cite[Thm.~1]{AC19}. As to the second statement, it is clear that $\delta$ is $C^2$ for $r>0$, hence we only need to check that $\delta''$ extends continuously at $r=0$ (from the right). For this purpose it can be assumed that $(\delta,\eta)$ lies within an arbitrarily small disk $D$ centered in $(\delta(0),\eta(0)=\delta(0))$.  Then, by Taylor expanding the right hand side of~\eqref{staticsystem}, we obtain
\[
\widehat{a}(\delta(r),\eta(r))\delta'(r)=\mathcal{R}(\delta(r),\eta(r))\frac{(\delta(r)-\eta(r))^2}{r}-\frac{4\pi G}{3}\mathcal{K}^2\delta(r)\eta(r) r,
\]
where $\mathcal{R}(\delta,\eta)$ is bounded in $D$ and thus in particular $\mathcal{R}(\delta(r),\eta(r))$ is bounded for $r$ small. Dividing by $r$ and taking the limit $r\to 0^+$, using that $(\delta(r)-\eta(r))^2/r^2\to 0$ as $r\to 0^+$, we find
\[
\delta''(0)=\lim_{r\to0^+}\frac{\delta'(r)}{r}=-\frac{4\pi G}{3}\mathcal{K}^2\frac{\delta(0)^2}{a(\delta(0),\delta(0))}.
\]
Moreover by the second equation in~\eqref{staticsystem} we have
\[
\frac{\eta'(r)}{r}=\frac{3\delta(r)}{r^2}-\frac{9\int_0^r\delta(s)s^2}{s^5}.
\]
Replacing $\delta(r)=\delta(0)+\frac{1}{2}\delta''(0)r^2+o(r^2)$, carrying out the integral and taking the limit $r\to 0^+$ we find
\[
\eta''(0)=\lim_{r\to0+}\frac{\eta'(r)}{r}=\frac{3}{5}\delta''(0).
\]
(ii) By (i), $v(r)=\eta(r)-\delta(r)$ is increasing for small $r>0$, hence $v(r)>0$ holds for small radii. Let $r_*>0$ be the supremum radius in the interval $(0,R)$ for which $v(r)>0$ for $r\in(0,r_*)$. If $r_*<R$, then $v(r_*)=0$ and $v'(r_*)\leq 0$. However this is not possible because
\begin{align*}
\lim_{r\to r_*^-}\widehat{a}(\delta(r),\eta(r))v'(r)&=-\widehat{a}(\delta(r_*),\eta(r_*))\left(3\frac{\eta(r_*)-\delta(r_*)}{r_*}\right)\\
&\quad-\frac{\widehat{b}(\delta(r_*),\eta(r_*))}{r_*}(\eta(r_*)-\delta(r_*))+\frac{4\pi}{3}\mathcal{K}^2r_*\,\eta(r_*)\,\delta(r_*)\\
&=\frac{4\pi G}{3}\mathcal{K}^2r_*\,\eta(r_*)^2>0,
\end{align*}
where for the second equality we used that $\delta(r_*)=\eta(r_*)$. Hence $\eta(r)>\delta(r)$, and thus $\eta'(r)<0$, must hold for $r\in(0,R)$. 

(iii) Since the weak Baker-Ericksen inequality means that~\eqref{BEineq} is satisfied for all $\eta\geq\delta$, then~\eqref{pradptan*} follows from (ii). 
\end{proof}
{\it Remark.} Thanks to the result in Lemma~\ref{propostatic}(iii), when the constitutive function satisfies the weak Baker-Ericksen inequality the positivity of the tangential pressure in the interior of static balls follows from the positivity of the radial pressure. 
%

In Section~\ref{staticsec} it will be shown that static self-gravitating elastic balls can be constructed by truncating strongly regular solutions of~\eqref{staticsystem} at the first radius $R>0$  where $p_\mathrm{rad}(r)=\mathcal{P}+\prad(\delta(r),\eta(r))$ vanishes. 
It is convenient to distinguish two types of static self-gravitating balls.
\begin{definition}\label{typeballdef}
Let $(\rho,p_\mathrm{rad},p_\mathrm{tan})$ be a static self-gravitating elastic ball obtained by truncating a strongly regular solution of~\eqref{staticsystem} with maximal interval of existence $[0,R_\mathrm{max})$. If $R_\mathrm{max}<\infty$, $(\rho,p_\mathrm{rad},p_\mathrm{tan})$ is said to be of type $\mathrm{A}$, while if $R_\mathrm{max}=\infty$ then $(\rho,p_\mathrm{rad},p_\mathrm{tan})$ is said to be of type $\mathrm{B}$.
\end{definition}
{\it Remark.} 
The distinction between types $\mathrm{A}$ and $\mathrm{B}$ may be important to determine the number of shells which can surround a static elastic ball if the different bodies are all made of the same elastic material. In~\cite{AC18} it was shown that in the particular case of the Seth elastic model, for which all static self-gravitating balls are of type $\mathrm{B}$, there is no limit on the number of shells that can form around the ball. An interesting open question is whether the same is possible for static self-gravitating elastic balls of type $\mathrm{A}$. In the polytropic fluid case, only type $\mathrm{A}$ are admissible and no shell can form around static self-gravitating balls, because the interior pressure is a decreasing function of the radius.

\subsection{Final comments on the constitutive function of elastic bodies}\label{finalcom}
This section contains a list of three properties which, in the author opinion, should be satisfied by elastic material models in the applications to the problem of self-gravitating bodies. The following list of conditions is inspired partially by physical considerations and mostly by the wish to avoid unnecessary mathematical complications not related to the problem. 
\begin{enumerate}
\item {\it The constitutive function should not depend on additional material parameters other than the bulk modulus $\kappa$ and the Poisson ratio $\nu$}. Some of the elastic models used in materials science, e.g. the Signorini model, do not satisfy this condition, see~\cite{ACA}. It should be emphasized that the extra parameters in these models are not proper material constants because their value depends on having assumed a specific constitutive function for the body. While considering additional phenomenological parameters may be needed for engineering applications, it does not seem a priority at the present stage in the analysis of self-gravitating elastic bodies.
\item {\it The constitutive function should be globally strictly hyperbolic.}  When the material model is not globally strictly hyperbolic the existence theory for the equations of motion of elastic bodies becomes considerably more technical, since one has to show that solutions launched by initial data in the hyperbolicity region $\mathcal{H}=\{(\delta,\eta):\partial_\delta\prad(\delta,\eta)>0\}$ of the state space remain in this region for all times. As the region $\mathcal{H}$ is known explicitly only for rather simple constitutive functions, this task is often very hard, if not impossible, to be done analytically.  It turns out that many popular elastic models found in the literature, such as the Saint Venant-Kirchhoff model, the Hadamard model, and others, do not satisfy this condition, see~\cite{AC19}.
\item {\it The constitutive function should satisfy the Baker-Ericksen inequality and (in the hyperelastic case) the non-negative stored energy condition}. The are many constitutive inequalities proposed in the literature, some of which however have unclear physical interpretation (e.g., polyconvexity~\cite{Ball}). An important condition on elastic constitutive functions, and with a clear physical meaning, is the Baker-Ericksen inequality, see Definition~\ref{BEdef}. This inequality has already played a role in Lemma~\ref{propostatic} and will show up again in the following sections. 
\end{enumerate}

In the next section a new four-parameters family of elastic models is introduced that satisfies the properties 1-3 and which at the same time may be seen as a natural generalisation of the polytropic fluid constitutive function~\eqref{constpolyfluid}.

 \section{Polytropic hyperelastic constitutive function}\label{polysec}
The ultimate goal of this section is to justify the definition of polytropic elastic constitutive function for spherically symmetric bodies, see Definition~\ref{polytropicdef} below. To this regard we remark that the polytropic fluid constitutive function~\eqref{constpolyfluid} is uniquely characterised by the scaling invariance property 
\begin{equation}\label{scaleinvfluid}
\varepsilon^{1-\gamma}\widehat{p}\,'(\varepsilon\delta)=\widehat{p}\,'(\delta), \quad\text{for all $\varepsilon,\gamma>0$}.
\end{equation}
The identity~\eqref{scaleinvfluid} entails that the Euler equations are scale invariant and thus it is the reason behind the existence of self-similar solutions. 
As~\eqref{scaleinvfluid} plays a fundamental role in the mathematical and physical study of fluid dynamics, it is natural to look for an elastic constitutive function that satisfies a similar scaling property.
The minimal scaling condition on $(\prad,\ptan)$ required for the existence of self-similar solutions to~\eqref{sssystem2} (with $f\equiv 0$) is
\begin{equation}\label{scaleinvprop}
\varepsilon^{1-\gamma}\,\widehat{a}(\varepsilon\delta,\varepsilon\eta)=\widehat{a}(\delta,\eta),\quad \varepsilon^{1-\gamma}\,\widehat{b}(\varepsilon\delta,\varepsilon\eta)=\widehat{b}(\delta,\eta),\ \text{for all $\varepsilon,\gamma>0$},
\end{equation}
where $\widehat{a}(\delta,\eta)$ and $\widehat{b}(\delta,\eta)$ are given by~\eqref{abc}. Another property of the polytropic fluid constitutive function that we want to generalize to the elastic case is the constant boundary shear condition, see Definition~\ref{CBSdef}. 

\begin{definition}
An elastic constitutive function $(\prad,\ptan)$ for spherically symmetric bodies is said to be a generalized polytropic constitutive function if:
\begin{enumerate}
\item $(\prad,\ptan)$ is hyperelastic, globally strictly hyperbolic and satisfies~\eqref{scaleinvprop};
\item $\prad$ satisfies the CBS condition~\eqref{CBSeq}, for some $y_\mathrm{b}\geq0$.
\end{enumerate}
An elastic ball with constitutive function $(\prad,\ptan)$ satisfying the conditions 1-2 and with reference pressure $\mathcal{P}=-\prad(y_\mathrm{b},1)$ will be called a generalized polytropic elastic ball.
\end{definition}

\begin{proposition}\label{scaleinvpro}
 Let $f_\mathrm{rad},f_\mathrm{tan},g:(0,\infty)\to\R$ be $C^2$ functions satisfying 
\begin{equation}\label{condF}
f_\mathrm{rad}(1)=f_\mathrm{tan}(1)=0,\quad f'_\mathrm{rad}(1)= f'_\mathrm{tan}(1)=\frac{1-\nu}{1+\nu},\quad g(1)=0,\ g'(1)=1.
\end{equation}
Then 
\begin{subequations}\label{scaleinvconsteq}
\begin{align}
&\kappa^{-1}\prad(\delta,\eta)=3 f_\mathrm{rad}(\delta/\eta)\eta^{\gamma}+g(\eta),\\
&\kappa^{-1}\ptan(\delta,\eta)=-\textfrac{3}{2} f_\mathrm{tan}(\delta/\eta)\eta^{\gamma}+g(\eta)-\textfrac{3}{2}(\eta-\delta)g'(\eta)
\end{align}
\end{subequations}
define an elastic constitutive function for spherically symmetric bodies that satisfies~\eqref{scaleinvprop}. Conversely, any such constitutive function must be of the form~\eqref{scaleinvconsteq} for some $C^2$ functions $f_\mathrm{rad},f_\mathrm{tan},g:(0,\infty)\to\R$ satisfying~\eqref{condF}.
\end{proposition}
\begin{proof}
It is straightforward to verify that the functions~\eqref{scaleinvconsteq} satisfy the properties in Definition~\ref{elasticmaterialmodeldef}, as well as the scaling identity~\eqref{scaleinvprop}, for all given $C^2$ functions $f_\mathrm{rad},f_\mathrm{tan}, g$ verifying~\eqref{condF}. Conversely, assume that $(\prad,\ptan)$ is a constitutive function that satisfies~\eqref{scaleinvprop}. Then
\[
\left(\frac{d}{d\varepsilon}\varepsilon^{1-\gamma}\,\widehat{a}(\varepsilon\delta,\varepsilon\eta)\right)_{\varepsilon=1}=(1-\gamma) \widehat{a}(\delta,\eta)+\delta\partial_\delta \widehat{a}(\delta,\eta)+\eta\partial_\eta \widehat{a}(\delta,\eta)=0,
\]
hence $\widehat{a}(\delta,\eta)=3\kappa(1-\nu)\delta^{\gamma-1}h(\delta/\eta)/(1+\nu)$, where $h$ is an arbitrary $C^1$ function which, by the first equation in~\eqref{lincomp1}, satisfies $h(1)=1$. It follows that 
\[
\kappa^{-1}\prad(\delta,\eta)=3 f_\mathrm{rad}(\delta/\eta)\eta^{\gamma}+g(\eta),
\]
where
\[
f_\mathrm{rad}(y)=\frac{1-\nu}{1+\nu}\int_1^{y}z^\gamma h(z)\,dz,\quad g(\eta)=\kappa^{-1}\prad(\eta,\eta).
\]
Thus $f_\mathrm{rad}(1)=0$, $f'_\mathrm{rad}(1)=\frac{1-\nu}{1+\nu}$ and, by the second equation in~\eqref{lincomp1}, $g'(1)=1$. 
Likewise
\[
\left(\frac{d}{d\varepsilon}\varepsilon^{1-\gamma}\,\widehat{b}(\varepsilon\delta,\varepsilon\eta)\right)_{\varepsilon=1}=(1-\gamma) \widehat{b}(\delta,\eta)+\delta\partial_\delta \widehat{b}(\delta,\eta)+\eta\partial_\eta \widehat{b}(\delta,\eta)=0
\]
implies $\widehat{b}(\delta,\eta)=3\kappa(1-\nu)\delta^{\gamma-1} s(\delta/\eta)/(1+\nu)$, where $s$ is $C^1$ and $s(1)=0$ by~\eqref{bzero}. By the definition~\eqref{abc} of $\widehat{b}$ this entails that $\ptan$ has the form stated in~\eqref{scaleinvconsteq} with 
\[
f_\mathrm{tan}(y)=[3\gamma(1-y)-2]f_\mathrm{rad}(y)+3y(y-1)f'_\mathrm{rad}(y)+\frac{1-\nu}{1+\nu}(y-1)y^{\gamma-1} s(y).
\]
As $f_\mathrm{rad}(1)=0$ and $f'_\mathrm{tan}(1)=f'_\mathrm{rad}(1)$, the proof is completed.
\end{proof}
The elastic constitutive function~\eqref{scaleinvconsteq} contains the three arbitrary functions $f_\mathrm{rad}, f_\mathrm{tan},g$ and thus comprises a huge class of models.
However using Lemma~\ref{dyneqlemma} with $q(\eta)=g(\eta)+1-\eta$ we obtain that all these constitutive functions are dynamically equivalent when $f_\mathrm{rad}$ and $f_\mathrm{tan}$ are fixed. It is easy to see that there is at most one representative in this dynamical equivalent class that could satisfy the constant shear condition. In fact, assuming that there exists $y_\mathrm{b}$ such that $y_\mathrm{b}\partial_\delta\prad(y_\mathrm{b}\eta,\eta)+\partial_\eta\prad(y_\mathrm{b}\eta,\eta)=0$, for all $\eta>0$, then necessarily $g'(\eta)=-3\gamma f_\mathrm{rad}(y_\mathrm{b})\eta^{\gamma-1}$ and thus the conditions $g(1)=0, g'(1)=1$ give
\begin{equation}\label{ybeq}
f_\mathrm{rad}(y_\mathrm{b})=-(3\gamma)^{-1},\quad g(\eta)=\frac{\eta^\gamma-1}{\gamma}.
\end{equation} 
Hence any generalised elastic polytropic constitutive function has to have to form
\begin{subequations}\label{scaleinvconsteq2}
\begin{align}
&\kappa^{-1}\prad(\delta,\eta)=3 f_\mathrm{rad}(\delta/\eta)\eta^{\gamma}+\frac{\eta^{\gamma}-1}{\gamma},\\
&\kappa^{-1}\ptan(\delta,\eta)=-\frac{3}{2}(f_\mathrm{tan}(\delta/\eta)+1-\delta/\eta)\eta^{\gamma}+\frac{\eta^{\gamma}-1}{\gamma},
\end{align}
\end{subequations}
where  $\gamma> 0$ and $f_\mathrm{rad},f_\mathrm{tan}$ are $C^2$ functions that satisfy~\eqref{condF}. Furthermore the constitutive function~\eqref{scaleinvconsteq2} is globally strictly hyperbolic if and only if $f'_\mathrm{rad}(y)>0$, for all $y>0$, and thus if a solution of $f_\mathrm{rad}(y_\mathrm{b})=-(3\gamma)^{-1}$ exists then it is necessarily unique. We also obtain that the reference pressure of generalized polytropic elastic balls must be given by $\mathcal{P}=-\prad(y_\mathrm{b},1)=\kappa/\gamma$, as in the polytropic fluid case, see~\eqref{constpolyfluid}.

The number of free functions in~\eqref{scaleinvconsteq2} can be further reduced by imposing that the constitutive function should be hyperelastic.
\begin{proposition}\label{polypro}
The constitutive function~\eqref{scaleinvconsteq2} is hyperelastic if and only if 
\begin{equation}\label{hypercond}
f_\mathrm{tan}(y)=f_\mathrm{rad}(y)+3(1-\gamma) y \int_1^y\frac{f_\mathrm{rad}(s)}{s^2}\,ds.
\end{equation}
When~\eqref{hypercond} holds, the stored energy function is given by
\begin{equation}\label{scaleinvariantw}
\kappa^{-1}\widehat{w}(\delta,\eta)=\left\{\begin{array}{ll}{\displaystyle \eta^{\gamma-1}\left(3S(\delta/\eta)-\frac{1}{\gamma}(\delta/\eta)^{-1}+\frac{1}{\gamma-1}\right)+\frac{\delta^{-1}}{\gamma}-\frac{1}{(\gamma-1)} }& 
\text{for $\gamma\neq1$}\\[0.5cm]
3S(\delta/\eta)-(\delta/\eta)^{-1}-\log(\delta/\eta) +\log\delta+\delta^{-1}& \text{for $\gamma=1$}\end{array}\right.
\end{equation}
where
\[
S(y)=\int_1^y\frac{f_\mathrm{rad}(s)}{s^2}\,ds.
\]
Conversely, given any $C^3$ function $S:(0,\infty)\to\R$ such that
\begin{equation}\label{condG}
S(1)=S'(1)=0,\quad S''(1)=\frac{1-\nu}{1+\nu},
\end{equation} 
\eqref{scaleinvariantw} is the stored energy function for the elastic constitutive function~\eqref{scaleinvconsteq2} where $f_\mathrm{tan}(y)$ is given by~\eqref{hypercond} and 
\begin{equation}\label{Fhyper}
f_\mathrm{rad}(y)=y^2S'(y).
\end{equation}
\end{proposition}
\begin{proof}
Using~\eqref{equivhyper}, it is easy to show that the constitutive equation~\eqref{scaleinvconsteq} is hyperelastic if and only if~\eqref{hypercond} holds. The stored energy function is given by
\begin{align}
\kappa^{-1}\widehat{w}({\delta,\eta})&=\kappa^{-1}\int_\eta^\delta \frac{\prad(s,\eta)}{s^2}\,ds+\kappa^{-1}W(\eta)\nonumber\\
&=\eta^{\gamma-1}\left(3S(\delta/\eta)-\frac{1}{\gamma}(\delta/\eta)^{-1}+\frac{1}{\gamma}\right)+\frac{\delta^{-1}-\eta^{-1}}{\gamma}+ \kappa^{-1}W(\eta),\label{test}
\end{align}
where, imposing $\widehat{w}(1,1)=\prad(1,1)=\ptan(1,1)=0$ and $\ptan(\eta,\eta)=\prad(\eta,\eta)$, the integration function $W$ is given by
\begin{align}\label{W}
\kappa^{-1}W(\eta)=\left\{\begin{array}{ll}{\displaystyle \frac{1}{\gamma}\left(\frac{\eta^{\gamma-1}-1}{\gamma-1}+\eta^{-1}-1\right)}&\text{if $\gamma\neq 1$}\\[0.5cm]
\log\eta+\eta^{-1}-1& \text{if $\gamma=1$}\end{array}\right..
\end{align}
Replacing~\eqref{W} into~\eqref{test} gives~\eqref{scaleinvariantw}. Finally, given a $C^3$ function $S$ satisfying~\eqref{condG} it is straightforward to verify that the elastic constitutive function associated to the stored energy~\eqref{scaleinvariantw} has precisely the form~\eqref{scaleinvconsteq} where $f_\mathrm{rad},f_\mathrm{tan}$ satisfying~\eqref{condF} are given by~\eqref{Fhyper} and~\eqref{hypercond}.
\end{proof}

So far it has been proved that the stored energy function of any generalized polytropic elastic constitutive function for spherically symmetric bodies is to be of the form~\eqref{scaleinvariantw} for some function $S$ that satisfies~\eqref{condG} and $(y^2S'(y))'>0$, for all $y>0$. The function $S$ will be called the shear function, as it expresses how the stored energy function depends on the shear variable $y=\delta/\eta$.  Next a specific choice will be made for the shear function in such a way that the only freedom left in its definition
consists in fixing the value of a single parameter $\beta$.  

We begin by choosing the shear function to be a linear combination of power-laws, that is 
\[
S(y)=\sum_{j=1}^n\alpha_j y^{\beta_j}.
\] 
This choice seems rather natural for a polytropic model and agrees with the general form of most stored energy functions found in the literature.
By~\eqref{condG} the coefficients $\alpha_1,\dots,\alpha_n$ must satisfy the linear system
\[
\sum_{j=1}^n\alpha_j =0,\quad \sum_{j=1}^n\alpha_j \beta_j=0,\quad \sum_{j=1}^n\alpha_j\beta_j(\beta_j-1) =\frac{1-\nu}{1+\nu}.
\]
If this system has more than one solution, the stored energy function~\eqref{scaleinvariantw}  would depend on additional dimensionless constants besides the Poisson ratio $\nu$, thereby violating Property 1 in Section~\ref{finalcom}. Hence we set $n=3$ and thus obtain the following shear function
\begin{equation}\label{Sgeneral}
S(y)=\frac{1-\nu}{1+\nu}\frac{(\beta_2-\beta_3)y^{\beta_1}-(\beta_1-\beta_3)y^{\beta_2}+(\beta_1-\beta_2)y^{\beta_3}}{(\beta_1-\beta_2)(\beta_1-\beta_3)(\beta_2-\beta_3)},
\end{equation}
where $\beta_1,\beta_2,\beta_3$ are distinct real numbers. 
Replacing~\eqref{Sgeneral} in~\eqref{scaleinvariantw}$_{\gamma\neq1}$ we find that
\begin{equation}\label{prop2}
\widehat{w}(\delta,\eta)= \widehat{w}_\mathrm{pf}(\delta),\text{ for $\nu=1/2$ and $(\beta_1,\beta_2,\beta_3)=(\gamma-1,-1,0)$ (or permutations),}
\end{equation}
where $\widehat{w}_\mathrm{pf}$ is the stored energy function of polytropic fluids, see~\eqref{polystored}. Moreover~\eqref{scaleinvconsteq2} gives
\begin{equation}\label{fradpolyGEN}
f_\mathrm{rad}(y)=y^2S'(y)=\frac{1-\nu}{1+\nu}\frac{\beta_1(\beta_2-\beta_3)y^{\beta_1+1}-\beta_2(\beta_1-\beta_3)y^{\beta_2+1}+\beta_3(\beta_1-\beta_2)y^{\beta_3+1}}{(\beta_1-\beta_2)(\beta_1-\beta_3)(\beta_2-\beta_3)}.
\end{equation}
From this expression it is clear that proving the hyerbolicity condition $f'_\mathrm{rad}(y)>0$, for all $y>0$, and solving~\eqref{ybeq} for the boundary value $y_\mathrm{b}$ of the shear variable is not possible without imposing further conditions on the parameters $\beta_1,\beta_2,\beta_3$. In the following we choose to simplify the model by fixing the value of two of the parameters $\beta_1,\beta_2,\beta_3$ and letting only one be free.
The final condition that we require on the stored energy function is that it should reduce to the polytropic fluid stored energy function in a suitable limit when $\nu=1/2$. 
In view of~\eqref{prop2} this condition is achieved by choosing the fixed parameters in the shear function to be $-1$ and $0$. Denoting the free parameter as $\beta-1$, we obtain the following final form of the shear function
\begin{equation}\label{functionG}
S(y)=\left\{\begin{array}{ll}{\displaystyle \frac{1-\nu}{1+\nu}\left(\frac{y^{\beta-1}}{\beta(\beta-1)}+\frac{y^{-1}}{\beta}-\frac{1}{\beta-1}\right)} &\text{if $\beta\neq 1$}\\[0.5cm]
 {\displaystyle \frac{1-\nu}{1+\nu}\left( \log y+y^{-1}-1\right)} &\text{if $\beta=1$}\end{array}\right.,
\end{equation}
where $\beta\neq 0$; the expression for $\beta=1$ is obtained in the limit of the shear function for $\beta\neq 1$.
With this choice of the shear function, \eqref{fradpolyGEN} simplifies to
\begin{equation}\label{fradpoly}
f_\mathrm{rad}(y)=y^2S'(y)=\frac{1-\nu}{1+\nu}\frac{y^{\beta}-1}{\beta},
\end{equation}
and thus the CBS condition is satisfied with
\begin{equation}\label{ystarpoly}
y_\mathrm{b}=\left[1-\frac{\beta}{3\gamma}\left(\frac{1+\nu}{1-\nu}\right)\right]^{1/\beta},
\end{equation}
provided the expression within the square brackets is non-negative, that is
\begin{equation}\label{ystarpos}
\beta\leq3\frac{1-\nu}{1+\nu}\gamma.
\end{equation}
Note that
$0\leq y_\mathrm{b}<1$, for all $\gamma>0$, $\beta\neq 0$, and that $y_\mathrm{b}=0$ in the fluid limit $(\nu,\beta)=(1/2,\gamma)$.


\begin{table}[h!]
  \begin{center}
\begin{tabular}{l|c}
\hline\\
$ \gamma,\beta$ & $\kappa^{-1}\widehat{w}(\delta,\eta)$\\[0.5cm]
\hline\\
$\gamma\neq1$ & \multirow{2}{*}{${\displaystyle \eta^{\gamma-1}\Big(\frac{3(1-\nu)}{\beta(\beta-1)(1+\nu)}(\delta/\eta)^{\beta-1}+\Big(\frac{3(1-\nu)}{\beta(1+\nu)}-\frac{1}{\gamma}\Big)(\delta/\eta)^{-1}+\frac{1}{\gamma-1}-\frac{3(1-\nu)}{(\beta-1)(1+\nu)}\Big)}$}\\[0.7cm]
$\beta\neq 1$ &\hspace{-10cm} ${\displaystyle+\frac{\delta^{-1}}{\gamma} -\frac{1}{\gamma-1}}$ \\[0.5cm]
\hline\\
$\gamma=1$ &  \multirow{2}{*}{${\displaystyle \frac{3(1-\nu)}{\beta(\beta-1)(1+\nu)}(\delta/\eta)^{\beta-1}+\left(\frac{3(1-\nu)}{\beta(1+\nu)}-1\right)(\delta/\eta)^{-1}-\frac{3(1-\nu)}{(1+\nu)(\beta-1)}-\log(\delta/\eta)}$}\\[0.7cm]
$\beta\neq 1$ &\hspace{-10cm} ${\displaystyle +\delta^{-1}+\log\delta }$\\ [0.5cm]
\hline\\
$\gamma\neq 1$ &  \multirow{2}{*}{${\displaystyle\eta^{\gamma-1}\Big(\Big(\frac{3(1-\nu)}{1+\nu}-\frac{1}{\gamma}\Big)(\delta/\eta)^{-1}+\frac{3(1-\nu)}{1+\nu}(\log(\delta/\eta)-1)+\frac{1}{\gamma-1}\Big)+\frac{\delta^{-1}}{\gamma}-\frac{1}{\gamma-1}}$}\\
$ \beta=1$ & \\[0.5cm]
\hline\\
$\gamma=1$ &  \multirow{2}{*}{${\displaystyle\frac{2(1-2\nu)}{1+\nu}\left((\delta/\eta)^{-1}+\log(\delta/\eta)\right)-\frac{3(1-\nu)}{1+\nu}+\delta^{-1}+\log\delta}$} \\
$\beta=1$ & \\[0.5cm]
\hline
\end{tabular}
 \caption{Polytropic stored energy function. }
    \label{wpoly}
    \end{center}
\end{table}

Using~\eqref{functionG} in~\eqref{scaleinvariantw}, the stored energy function reads as in Table~\ref{wpoly}, depending on the values of the parameters $\gamma,\beta$.
The constitutive function $(\prad,\ptan)$ is given by
\begin{subequations}\label{pradptan}
\begin{align}
&\kappa^{-1}\prad(\delta,\eta)=\frac{3(1-\nu)}{\beta(1+\nu)}\big((\delta/\eta)^\beta-y_\mathrm{b}^\beta\big)\eta^\gamma-\frac{1}{\gamma}\\
&\kappa^{-1}\ptan(\delta,\eta)=\kappa^{-1}\prad(\delta,\eta)+3\big(1-\delta/\eta\big)Q(\delta/\eta)\eta^\gamma,
\end{align}
where 
\begin{equation}\label{Q}
Q(y)=\left\{\begin{array}{ll}{\displaystyle \frac{3(\beta-\gamma)(1-\nu)}{2\beta(\beta-1)(1+\nu)}\frac{1-y^\beta}{1-y}+\frac{3(\gamma-1)(1-\nu)}{2(\beta-1)(1+\nu)}-\frac{1}{2}} & \text{if $\beta\neq 1$}\\[0.5cm]
{\displaystyle \frac{3(\gamma-1)(1-\nu)}{2(1+\nu)}\frac{y\log y}{1-y}+\frac{3\gamma(1-\nu)}{2(1+\nu)}-\frac{1}{2}} & \text{if $\beta=1$}\end{array}\right..
\end{equation}
\end{subequations}
Moreover
\begin{subequations}\label{abpol}
\begin{align}
&\kappa^{-1}\widehat{a}(\delta,\eta)=\kappa^{-1}\partial_\delta\prad(\delta,\eta)=3\frac{1-\nu}{1+\nu}(\delta/\eta)^{\beta-1} \eta^{\gamma-1}\\[0.5cm]
&\kappa^{-1}\widehat{b}(\delta,\eta)=9\frac{1-\nu}{1+\nu}(\beta-\gamma)(1-\delta/\eta)^{-1}B(\delta/\eta)\eta^{\gamma-1},\\[0.5cm]
&\text{where}\quad B(y)=\left\{\begin{array}{ll}{\displaystyle \frac{y+(\beta-1) y^{1+\beta}-\beta y^{\beta}}{\beta(\beta-1)}} & \text{if $\beta\neq 1$}\\[0.5cm] y(y-\log y-1) & \text{if $\beta=1$}\end{array}\right..
\end{align}
\end{subequations}
In particular the constitutive function is globally strictly hyperbolic and, by Lemma~\ref{bfluid}, for $\beta=\gamma$ it is dynamically equivalent to a fluid constitutive function.

Our search for a hyperelastic constitutive function resembling the polytropic fluid model~\eqref{constpolyfluid} ends here and has led to the following definition.

\begin{definition}\label{polytropicdef}
Given $\kappa>0$, $\nu\in(-1,1/2]$,  $\gamma>0$, and $\beta\neq0$ satisfying~\eqref{ystarpos}, the function $(\prad,\ptan)$ defined by~\eqref{pradptan} is called polytropic elastic constitutive function for spherically symmetric bodies with bulk modulus $\kappa$, Poisson ratio $\nu$, polytropic exponent $\gamma$ and shear exponent $\beta$.
\end{definition} 

{\it Remark. }For $\nu=1/2$ and $\beta=\gamma$, the polytropic elastic constitutive function becomes the polytropic fluid model~\eqref{constpolyfluid}, for all $\gamma> 0$, while for $\beta=\gamma$ and for all $\nu\in (-1,1/2)$  it is dynamically equivalent to such fluid. 
In all other cases the constitutive function~\eqref{pradptan} is ``genuinely" elastic.


  
To conclude this section we show that the polytropic elastic constitutive function satisfies the weak Baker-Ericksen inequality and the weak non-negative stored energy condition for all admissible values of the parameters $\gamma,\beta$ and the strong version of these inequalities if and only if $\gamma,\beta$ verify an additional bound. 
\begin{proposition}\label{BEpro}
The polytropic elastic constitutive function~\eqref{pradptan} satisfies the weak Baker-Ericksen inequality and the weak non-negative stored energy condition for all  $\gamma>0$, and $\beta\neq0$ satisfying~\eqref{ystarpos}. If and only if in addition
\begin{equation}\label{condgamma}
\beta\geq\min\left(\gamma,3\frac{1-\nu}{1+\nu}\gamma-\frac{2(1-2\nu)}{1+\nu}\right)=\left\{\begin{array}{ll}\gamma & \text{if $\gamma\geq 1$}\\ {\displaystyle 3\frac{1-\nu}{1+\nu}\gamma-\frac{2(1-2\nu)}{1+\nu}} & \text{if $0<\gamma<1$}\end{array}\right.,
\end{equation}
then~\eqref{pradptan} satisfies the strong Baker-Ericksen inequality and the strong non-negative stored energy condition as well. 
\end{proposition}
\begin{proof}
For the polytropic elastic constitutive function the inequality~\eqref{BEineq} reads $Q(y)\geq0$, where $Q$ is defined in~\eqref{pradptan}.
We compute
\[
Q'(y)=\frac{3(\beta-\gamma)(1-\nu)}{2(1-y)^2(1+\nu)}\frac{B(y)}{y},
\]
where $B$ is defined in~\eqref{abpol}. A simple analysis shows that
\[
B(y)>0,\quad\text{for all $0<y\neq 1$},\quad B(1)=0.
\]
It follows that $Q(y)$ is strictly increasing for $\beta>\gamma$ and strictly decreasing for $\beta<\gamma$, while for $\beta=\gamma$ we have $Q(y)=\frac{1-2\nu}{1+\nu}\geq0$. Hence
$\inf_{y>0}Q(y)=\inf_{y\in(0,1)}Q(y)=\lim_{y\to 0^+}Q(y)$, for $\beta>\gamma$, while $\inf_{y>0}Q(y)=\lim_{y\to \infty}Q(y)$, and $\inf_{y\in(0,1)}Q(y)=\lim_{y\to 1}Q(y)=\frac{1-2\nu}{1+\nu}\geq0$, for $\beta<\gamma$. 
The proof of the claims on the weak and the strong Baker-Ericksen inequality can now easily be completed by using that
\[
\lim_{y\to 0^+}Q(y)=\left\{\begin{array}{ll} {\displaystyle \frac{3\gamma(1-\nu)}{2\beta(1+\nu)}-\frac{1}{2}}& \text{if $\beta>0$}\\[0.4cm] {\displaystyle \frac{\gamma-\beta}{\beta(\beta-1)}\cdot\infty }& \text{if $\beta<0$}\end{array}\right.,\quad
\lim_{y\to \infty}Q(y)=\left\{\begin{array}{ll} {\displaystyle \frac{3(\gamma-1)(1-\nu)}{2(\beta-1)(1+\nu)}-\frac{1}{2}}& \text{if $\beta<1$}\\[0.4cm] {\displaystyle \frac{\beta-\gamma}{\beta(\beta-1)}\cdot\infty }& \text{if $\beta>1$}\\[0.4cm]
{\displaystyle \frac{1-\nu}{1+\nu}(1-\gamma)\cdot\infty} & \text{if $\beta=1$.}
\end{array}\right..
\]
As to the claims on the weak and strong non-negative stored energy condition, we first observe that, by Lemma~\ref{wposlemma}, for all  $\gamma>0$, and $\beta\neq0$ satisfying~\eqref{ystarpos} there holds
\[
\widehat{w}(\delta,\eta)\geq \inf_{\delta>0}\widehat{w}(\delta,\delta),\quad \text{for all $\eta\geq \delta>0$}
\]
and if in addition $(\beta,\gamma)$ satisfy~\eqref{condgamma} there also holds
\[
\widehat{w}(\delta,\eta)\geq \inf_{\delta>0}\widehat{w}(\delta,\delta),\quad \text{for all $(\delta,\eta)\in (0,\infty)^2$}.
\]
As $\widehat{w}(\delta,\delta)=\widehat{w}_\mathrm{pf}(\delta)$ and $\widehat{w}_\mathrm{pf}(\delta)>\widehat{w}_\mathrm{pf}(1)=0$,  for all $0<\delta\neq 1$, see~\eqref{polystored}, the ``\,if\," part of the claim on the non-negative stored energy condition is proved. It remains to prove that when $(\beta,\gamma)$ do not satisfy~\eqref{condgamma} the strong non-negative stored energy condition is violated, namely, there exists $\eta< \delta$ such that $\widehat{w}(\delta,\eta)<0$.  We present the details only for $(\beta,\gamma)\neq (1,1)$, which is actually the less trivial case. Assume first $\beta<1$. Then~\eqref{condgamma} is violated for all $\gamma>1$ and for $\gamma<1$ when
\begin{equation}\label{condtemp}
\beta<3\frac{1-\nu}{1+\nu}\gamma-\frac{2(1-2\nu)}{1+\nu}.
\end{equation}
Moreover for all $\eta>0$ fixed we have
\[
\lim_{\delta\to\infty}\widehat{w}(\delta,\eta)=\eta^{\gamma-1}\left(\frac{1}{\gamma-1}-\frac{3(1-\nu)}{(\beta-1)(1+\nu)}\right)-\frac{1}{\gamma-1}:=F(\eta)
\]
and it is straightforward to verify that  for all $\gamma>1$ and $\beta<1$, or for all $\gamma<1$, $\beta<1$ satisfying~\eqref{condtemp}, $F(\eta)$ is negative when $\eta$ is small enough. In the case $\beta>1$,~\eqref{condgamma} is violated for $\gamma>\beta $. Under the latter condition, for all $\delta$ fixed we have
\[
\lim_{\eta\to 0^+}\widehat{w}(\delta,\eta)=\frac{\delta^{-1}}{\gamma}-\frac{1}{\gamma-1},
\]
which is negative for $\delta>(\gamma-1)/\gamma$. This concludes the proof of the proposition.
\end{proof}
The region in the $(\gamma,\beta)$ plane where~\eqref{pradptan} satisfies the strong Baker-Ericksen inequality and the strong non-negative energy stored condition for $\nu\in(-1,1/2)$ fixed is depicted in Figure~\ref{fig1}. In the limit $\nu\to 1/2^-$ this region reduces to the line $\beta=\gamma$.
\begin{figure}[ht!]
\begin{center}
\includegraphics[width=0.6\textwidth]{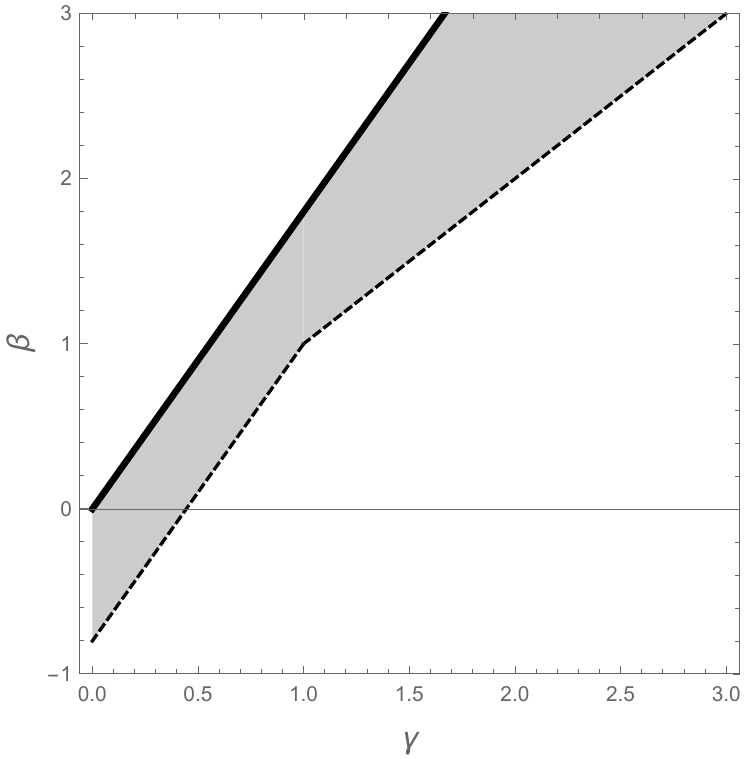}
\caption{The region on the $(\gamma,\beta)$ plane where the polytropic elastic constitutive function satisfies the strong Baker-Ericksen inequality and the strong non-negative stored energy condition for a fixed value of $\nu\in(-1,1/2)$ ($\nu=1/4$ in the picture). The dashed line is the curve~\eqref{condgamma}, the continuous line is the curve $\beta=3\frac{1-\nu}{(1+\nu)}\gamma$. For $\nu=1/2$ the two lines both reduce to $\beta=\gamma$. The weak Baker-Ericksen inequality and the weak non-negative stored energy condition are verified everywhere below the continuous line. }
\label{fig1}
\end{center}
\end{figure}

\section{Self-gravitating polytropic elastic balls}\label{staticsec}
The main topic of this section is the numerical analysis of the existence of self-gravitating polytropic elastic balls in static equilibrium or homologous motion. The constitutive function of the ball is then given by~\eqref{pradptan}, with $\gamma>0$, $\beta\neq 0$ satisfying~\eqref{ystarpos}, while the reference pressure in the equation of state is given by $\mathcal{P}=-\prad(y_\mathrm{b},1)=\kappa/\gamma$, where $y_\mathrm{b}$ is given by~\eqref{ystarpoly}. Hence the equation of state of the ball is 
\begin{subequations}\label{pressures}
\begin{equation}
p_\mathrm{rad}(t,r)=F_\mathrm{rad}\left(\frac{\rho(t,r)}{\mathcal{K}},\frac{m(t,r)}{\frac{4\pi}{3}\mathcal{K}r^3}\right), \quad p_\mathrm{tan}(t,r)=F_\mathrm{tan}\left(\frac{\rho(t,r)}{\mathcal{K}},\frac{m(t,r)}{\frac{4\pi}{3}\mathcal{K}r^3}\right),
\end{equation}
where
\begin{align}
&F_\mathrm{rad}(\delta,\eta)=\frac{\kappa}{\gamma}+\prad(\delta,\eta)=3\kappa\frac{1-\nu}{1+\nu}\left(\frac{(\delta/\eta)^\beta-y_\mathrm{b}^\beta}{\beta}\right)\eta^\gamma\\
&F_\mathrm{rad}(\delta,\eta)=\frac{\kappa}{\gamma}+\ptan(\delta,\eta)=3\kappa\frac{1-\nu}{1+\nu}\left[\left(\frac{(\delta/\eta)^\beta-y_\mathrm{b}^\beta}{\beta}\right)+\frac{1+\nu}{1-\nu}\big(1-\delta/\eta\big)Q(\delta/\eta)\right]\eta^\gamma
\end{align}
\end{subequations}
where $Q(y)$ is given by~\eqref{Q}.
From these expressions it is clear that the interior of the ball corresponds to the region where $y(t,r)>y_\mathrm{b}$, while the boundary is the first radius $r=R(t)$ at which $y(t,R(t))=y_\mathrm{b}$; recall that $y(t,0)=1$ and $y_\mathrm{b}\in [0,1)$. 
Moreover if (and only if) 
$\beta=3\frac{1-\nu}{1+\nu}\gamma$
we have $y_\mathrm{b}=0$, and thus the mass density vanishes on the boundary of the ball. In this case~\eqref{pressures} give that the radial and tangential pressure coincide on the boundary, that is to say, the boundary has zero shear stress. A special example of this situation, but not the only possible one, is when $\nu=1/2$ and $\beta=\gamma$, i.e., when the elastic ball is actually a polytropic fluid ball. 

\subsection{Static solutions}\label{staticsec}
Using~\eqref{abpol} in~\eqref{staticsystem} we obtain the following system on $(\delta,\eta)$:
\begin{subequations}\label{deltastat}
\begin{align}
&\delta'=\left(\frac{\delta}{\eta}\right)^{1-\beta}\left(\frac{3(\beta-\gamma)}{r}B(\delta/\eta)\eta-\theta r\eta^{2-\gamma}\delta\right),\\
&\eta'=-\frac{3}{r}(\eta-\delta),
\end{align}
where
\begin{equation}\label{deftheta}
\theta=\frac{4\pi G\mathcal{K}^2}{9\kappa}\frac{1+\nu}{1-\nu}.
\end{equation}
\end{subequations}
The system~\eqref{deltastat} is supplied with center data $\delta(0)=\eta(0)=\delta_c>0$, i.e., $y(0)=1$. If within the interval of existence of strongly regular solutions of~\eqref{deltastat} there is a first radius $R>0$ such that $y(R)=y_\mathrm{b}$, then a static self-gravitating polytropic elastic ball with radius $R$ forms. The positivity of the tangential pressure in the interior follows by the weak Baker-Ericksen inequality, see Lemma~\ref{propostatic}(iii).
%
%
The values of the parameters $\gamma,\beta$ for which it is found numerically that static self-gravitating polytropic elastic balls exist form the domain $\mathcal{O}$, depending on $\nu$, depicted in Figure~\ref{fig2}.  The particular shape of this region leads to the following conjecture:

{\bf Conjecture 1:} {\it There exists $\beta_\star=\beta_\star(\nu,\gamma)$ and $\gamma_\star=\gamma_\star(\nu)$ satisfying 
\begin{equation}\label{betastar}
\beta_\star(\nu,\gamma_\star(\nu))=3\frac{1-\nu}{1+\nu}\gamma_\star(\nu),\quad\text{ and }\quad\beta_\star(\nu,\gamma)<3\frac{1-\nu}{1+\nu}\gamma, \text{ for $0<\gamma<\gamma_\star(\nu)$},
\end{equation} 
such that static self-gravitating polytropic elastic balls exist if and only if 
\[
\mathrm{(a)}\ 0<\gamma\leq\gamma_\star(\nu) \text{ and }\beta< \beta_\star(\nu,\gamma),\quad\text{or}\quad 
\mathrm{(b)}\ \gamma> \gamma_\star(\nu)\text{ and }\beta\leq 3\frac{1-\nu}{1+\nu}\gamma.
\]
}

{\it Remark.} When $\beta=3\frac{1-\nu}{1+\nu}\gamma$, i.e., when the boundary of the polytropic elastic ball has zero shear, condition (a) is empty, while condition (b) becomes $\gamma> \gamma_\star(\nu)$. These static balls are necessarily of type $\mathrm{A}$. Moreover it is found numerically that $\gamma_\star$ is increasing, $\gamma_\star(\nu)\to 0$ as $\nu\to -1^+$ and $\gamma_\star(\nu)\to 6/5$ as $\nu\to 1/2^-$, see Figure~\ref{fig3}. In particular, when $\nu\to 1/2^-$ the conjecture reduces to $\gamma>6/5$, which is indeed the necessary and sufficient condition for the existence of static self-gravitating polytropic fluid balls.

\begin{figure}[ht!]
\begin{center}
\subfigure[$\nu=-0.5, \gamma_\star(\nu)\approx 0.50$]{
\includegraphics[width=0.47\textwidth,trim=0cm 0cm 0cm 0cm, clip ]{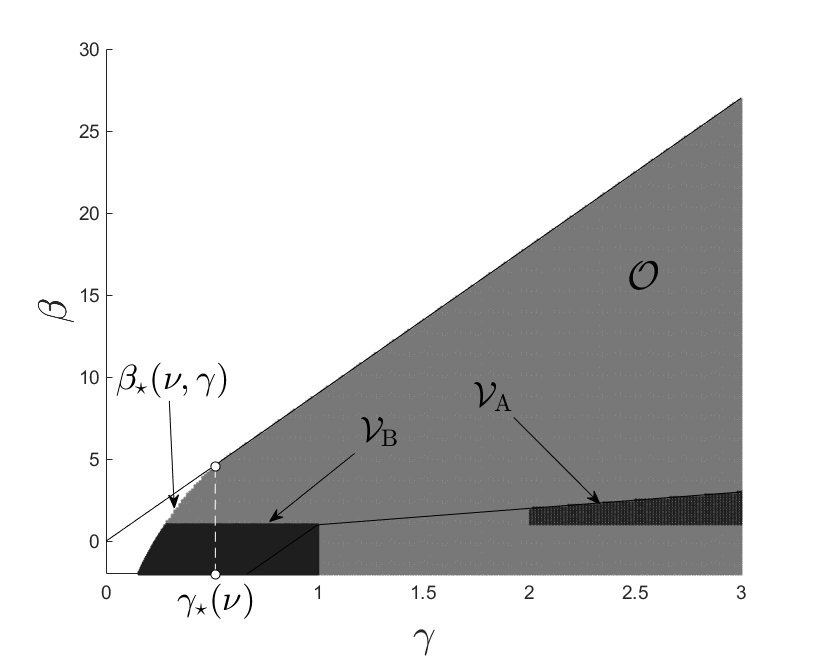}}\quad
\subfigure[$\nu=0, \gamma_\star(\nu)\approx 0.92$]{
\includegraphics[width=0.47\textwidth,trim=0cm 0cm 0cm 0cm, clip ]{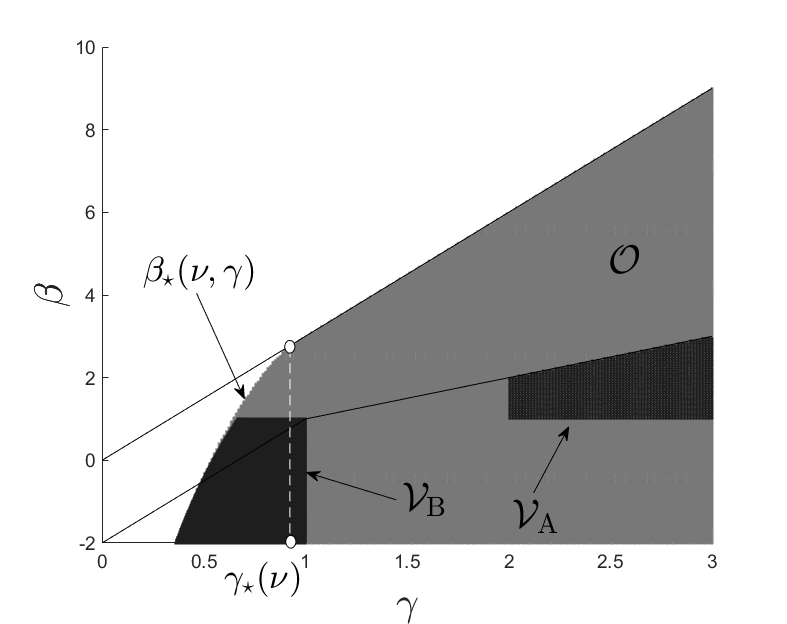}}\\
\subfigure[$\nu=0.25, \gamma_\star(\nu)\approx 1.08$]{\includegraphics[width=0.47\textwidth,trim=0cm 0cm 0cm 0cm, clip ]{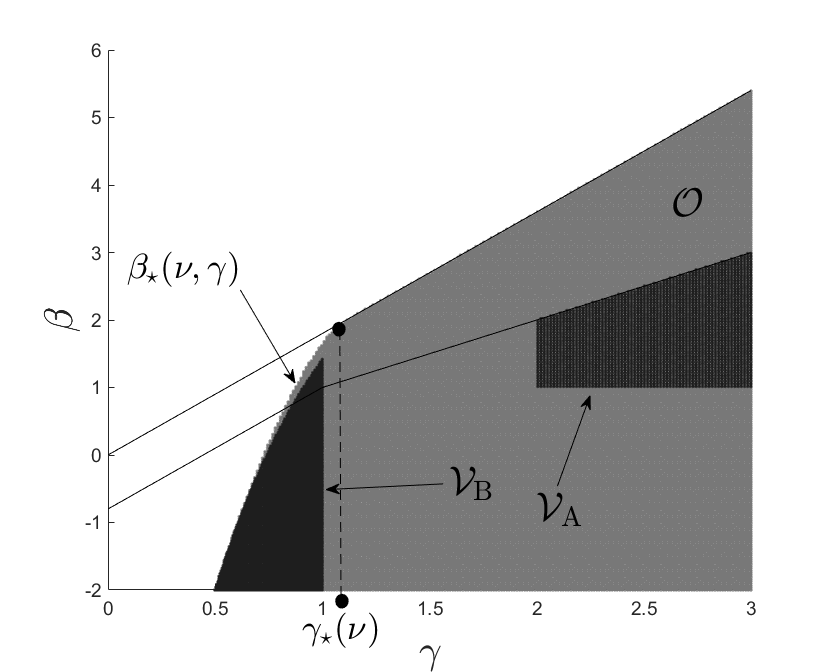}}\quad
\subfigure[$\nu=0.48,\gamma_\star(\nu)\approx 1.19$]{
\includegraphics[width=0.47\textwidth,trim=0cm 0cm 0cm 0cm, clip ]{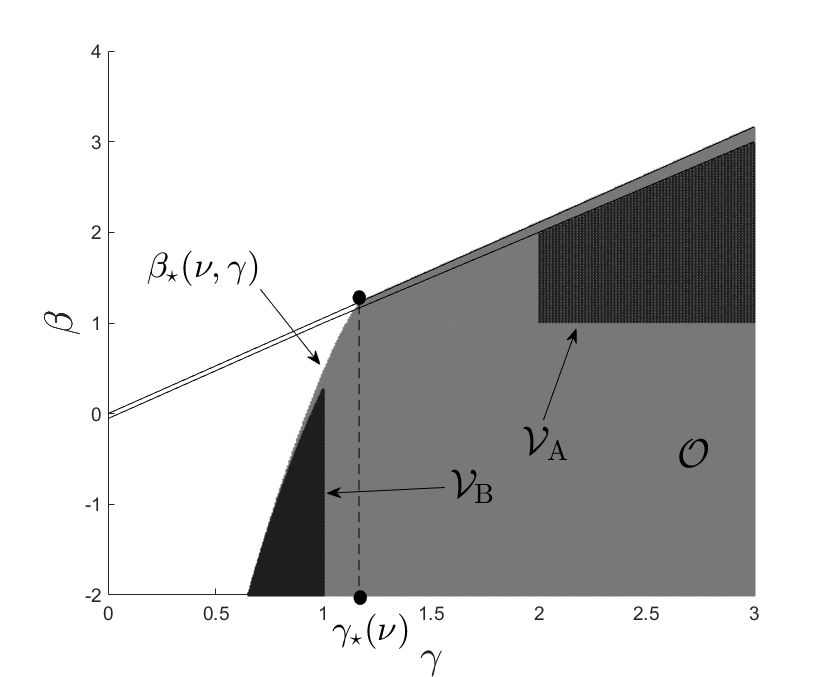}}
\caption{The region $\mathcal{O}$ in the $(\gamma,\beta)$ for which it is found numerically that static self-gravitating polytropic elastic balls exist. The darkest region is the domain $\mathcal{V}=\mathcal{V}_\mathrm{A}\cup\mathcal{V}_\mathrm{B}$ where existence is proved analytically in Theorem~\ref{maintheostatic}. The set $\mathcal{V}_\mathrm{A}$ is defined by~\eqref{condgammabetaI} and corresponds to type $\mathrm{A}$ balls;  the set $\mathcal{V}_\mathrm{B}$ is defined by~\eqref{condgammabetaII} and corresponds to type $\mathrm{B}$ balls. The upper continuous line is $\beta=3(1-\nu)/(1+\nu)\gamma$, which corresponds to polytropic elastic balls with zero boundary shear, while the lower continuous line is the curve in~\eqref{condgamma}. }
\label{fig2}
\end{center}
\end{figure}

\begin{figure}[ht!]
\begin{center}
\includegraphics[width=0.8\textwidth,trim=4cm 4cm 4cm 4cm, clip ]{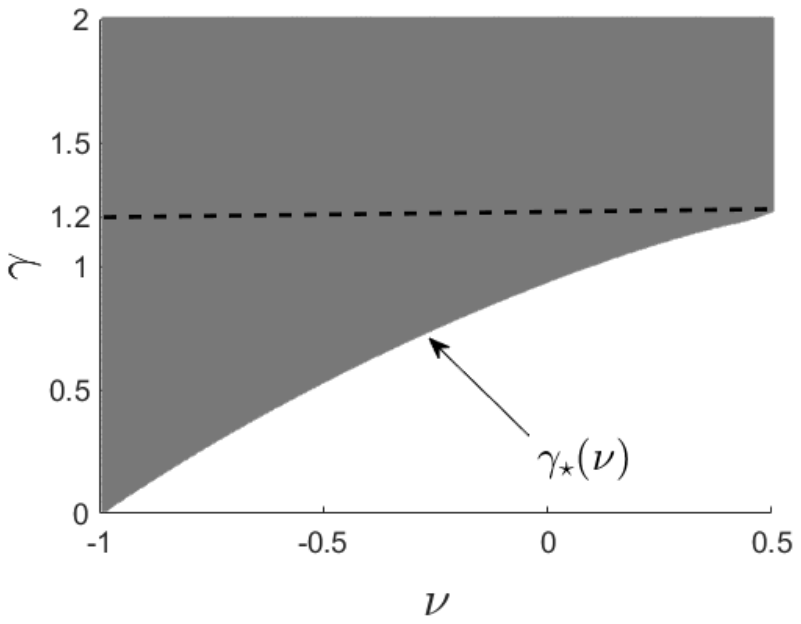}
\caption{The region on the $(\nu,\gamma)$ plane for which it found numerically that static self-gravitating polytropic elastic balls with zero shear at the boundary exist. These static balls are necessarily of type $\mathrm{A}$ and are characterized by the shear parameter $\beta=3\gamma(1-\nu)/(1+\nu)$.  }
\label{fig3}
\end{center}
\end{figure}


The the dark region $\mathcal{V}=\mathcal{V}_\mathrm{A}\cup\mathcal{V}_\mathrm{B}$ in Figure~\ref{fig2} contains the values of the parameters $\beta,\gamma$ for which the existence of static self-gravitating polytropic elastic balls is proved in the following theorem.
 
\begin{theorem}\label{maintheostatic}
Let $\delta_c>0$, $\kappa>0$ and $\nu\in (-1,1/2)$ be given. If
\begin{equation}\label{condgammabetaI}
\gamma>2 \quad\text{and}\quad 1<\beta\leq\gamma,
\end{equation}
there exists a unique static self-gravitating polytropic elastic ball of type $\mathrm{A}$ with central density $\delta(0)=\delta_c$. 
If 
\begin{equation}\label{condgammabetaII}
 (0<\gamma\leq \beta<1 \text{ or } \beta<\gamma\leq 1) \quad \text{and}\quad \frac{4-3\gamma}{3(2-\gamma)}<\left(1-\frac{\beta}{3\gamma}\frac{1+\nu}{1-\nu}\right)^{1/\beta},
\end{equation}
there exists a unique static self-gravitating polytropic elastic ball of type $\mathrm{B}$ with central density $\delta(0)=\delta_c$.
Moreover in both cases
\begin{equation}\label{es}
p_\mathrm{tan}(r)\geq p_\mathrm{rad}(r),\text{ for $r\in [0,R]$},
\end{equation}
where $R$ is the radius of the ball.
\end{theorem}

The proof of Theorem~\ref{maintheostatic} can be found in Appendix. It can be seen that the proved result is still quite far from the conjectured one and does not apply to the important case of polytropic elastic balls with zero boundary shear; this problem requires further investigation.


\subsection{Homologous solutions}\label{homosec}
In this final section we discuss briefly self-gravitating polytropic elastic balls in homologous motion.  As mentioned in Section~\ref{remhomo}, the solution $(\delta,\eta,u)$ of the system~\eqref{sssystem2} for such balls has the form 
\begin{subequations}\label{homansatz2}
\begin{equation}
u(t,r)=\frac{\dot{\omega}(t)}{\omega(t)}r,\quad \delta(t,r)=\frac{1}{\omega(t)^3}\delta_0\left(\frac{r}{\omega(t)}\right),\quad \eta(t,r)=\frac{1}{\omega(t)^3}\eta_0\left(\frac{r}{\omega(t)}\right),
\end{equation} 
where
\begin{equation}
\eta_0(r)=\frac{3}{r^3}\int_0^r\delta_0(s)s^2\,ds,\quad \delta_0(r)=\delta(0,r),
\end{equation}
for some $C^1$ function $\omega: \omega(0)=1$.
\end{subequations}
The system~\eqref{sssystem2} for solutions of the form~\eqref{homansatz2} reads
\begin{align}
\mathcal{K}z\delta_0(z)\omega(t)\ddot \omega(t)&=-\widehat{a}(\omega(t)^{-3}\delta_0(z),\omega(t)^{-3}\eta_0(z))\delta_0'(z)\nonumber\\
&\quad+\widehat{b}(\omega(t)^{-3}\delta_0(z),\omega(t)^{-3}\eta_0(z))\frac{\eta_0(z)-\delta_0(z)}{z}\nonumber\\
&\quad-\frac{4\pi G}{3}\mathcal{K}^2z
\delta_0(z)\eta_0(z)\omega(t)^{-1},\label{homologsys}
\end{align}
where $z=r/\omega(t)$.
Provided the functions $\widehat{a},\widehat{b}$ satisfy the scaling condition 
\[
\varepsilon^{-1/3}\,\widehat{a}(\varepsilon\delta,\varepsilon\eta)=\widehat{a}(\delta,\eta),\quad \varepsilon^{-1/3}\,\widehat{b}(\varepsilon\delta,\varepsilon\eta)=\widehat{b}(\delta,\eta),\ \text{for all $\varepsilon>0$,}
\]
the system~\eqref{homologsys} transforms into three separate equations on $\omega(t),\delta_0(z),\eta_0(z)$, namely
\begin{subequations}\label{homologsys2}
\begin{align}
&\omega(t)^2\ddot{\omega}(t)=\frac{4\pi G}{3}\mathcal{K}\alpha,\label{eqphi}\\
&\widehat{a}(\delta_0(z),\eta_0(z))\delta_0'(z)=\widehat{b}(\delta_0(z),\eta_0(z))\frac{\eta_0(z)-\delta_0(z)}{z}-\frac{4\pi G}{3}\mathcal{K}^2z\delta_0(z)(\eta_0(z)+\alpha),\label{eqdelta0pol}\\
&\eta_0'(z)=-\frac{3}{z}(\eta_0(z)-\delta_0(z))\label{eqeta0pol},
\end{align}
\end{subequations}
where $\alpha\neq 0$ is a (dimensionless) constant. 

In the following we study the system~\eqref{homologsys2} for the polytropic constitutive function $(\prad,\ptan)$ given by~\eqref{pradptan} with $\gamma=4/3$. We are interested in solutions describing self-gravitating balls with radius $R(t)$ and, as shown in Section~\ref{remhomo}, this implies $R(t)=Z\omega(t)$, for some constant $Z>0$. Moreover the reference pressure is given by $\mathcal{P}=-\prad(y_\mathrm{b},1)=\kappa/\gamma$, hence the equation of state of the ball is~\eqref{pressures}. In particular, $y(t,R(t))=y_\mathrm{b}$.

The equation for $\omega$ is the same as in the polytropic fluid case~\cite{GW,Makino2}. Positive solutions of~\eqref{eqphi} with data $(\omega(0)=1,\dot{\omega}(0))$ have one of the following behaviors: either (a) $\omega\in C^1$ for all $t\geq 0$ and $\omega(t)\to\infty$ as $t\to\infty$, or (b) $\omega\in C^1([0,T))$ and $\omega(t)\to 0$ as $t\to T^-$, for some $T>0$.  Without loss of generality we assume $\dot{\omega}(0)=0$, i.e., the ball is initially at rest, so that solutions of type (a) correspond to $\alpha>0$, while solutions of type (b) correspond to $\alpha<0$.
Hence the system~\eqref{deltahom} describes continuously expanding balls when $\alpha>0$ and continuously collapsing balls when $\alpha<0$, with $R(t)\to \infty$ as $t\to\infty$ in the former case and $R(t)\to 0$ as $t\to T^-$ in the latter.

Next consider the equations for $(\delta_0,\eta_0)$, specifically
\begin{subequations}\label{deltahom}
\begin{align}
&\delta_0'=\left(\frac{\delta_0}{\eta_0}\right)^{1-\beta}\left(\frac{3(\beta-4/3)}{z}B(\delta_0/\eta_0)\eta_0-\theta z\eta^{-1/3}\delta_0(\eta_0+\alpha)\right),\\
&\eta_0'=-\frac{3}{z}(\eta_0-\delta_0),
\end{align}
\end{subequations}
where $\theta$ is given by~\eqref{deftheta}. 
Strongly regular solutions of~\eqref{deltahom} are defined precisely as in the static case and satisfy $y_0(0)=1$, where $y_0(z)=\delta_0(z)/\eta_0(z)$. Moreover since the difference between~\eqref{deltahom} and~\eqref{deltastat} is a $O(z)$ term, it is clear that a unique strongly regular solution of~\eqref{deltahom} with center data $\delta_0^c=\delta_0(0)=\eta_0(0)$ exists in a maximal interval $[0,Z_\mathrm{max})$. As
\begin{subequations}\label{label}
\begin{align}
&p_\mathrm{rad}(t,r)=F_\mathrm{rad}(\delta(t,r),\eta(t,r))=\omega(t)^{-4}F_\mathrm{rad}(\delta_0(r/\omega(t)),\eta_0(r/\omega(t)),\\
&p_\mathrm{tan}(t,r)=F_\mathrm{tan}(\delta(t,r),\eta(t,r))=\omega(t)^{-4}F_\mathrm{tan}(\delta_0(r/\omega(t)),\eta_0(r/\omega(t)),
\end{align}
\end{subequations}
the constant $Z$ in the radius $R(t)=Z\omega (t)$ of the ball is given by 
the first value of $z\in (0,Z_\mathrm{max})$ at which $F_\mathrm{rad}(\delta_0(z),\eta(z))=0$, that is $y_0(Z)=y_\mathrm{b}$, where the boundary shear $y_\mathrm{b}$ is given by
\begin{equation}
y_\mathrm{b}=\left[1-\frac{\beta}{4}\left(\frac{1+\nu}{1-\nu}\right)\right]^{1/\beta},
\end{equation}
see~\eqref{ystarpoly}. As in the static case, we can distinguish between type $\mathrm{A}$ homologous balls, for which $0<Z<Z_\mathrm{max}<\infty$, and type $\mathrm{B}$ homologous balls, for which $0<Z<Z_\mathrm{max}=\infty$; in the case of polytropic elastic balls with zero boundary shear (including fluid balls) only type $\mathrm{A}$ solutions are admissible.

The existence of a radius $Z$ with the properties above has been investigated numerically in the case of collapsing balls with zero shear at the boundary, i.e., $\beta= 4(1-\nu)/(1+\nu)$; the results are summarized in Figure~\ref{homofig}. It is found that  there exists $\delta_\star(\alpha,\nu)$ such that collapsing homologous self-gravitating polytropic elastic balls with zero boundary shear exist if and only if $\delta_0(0)\geq\delta_\star(\alpha,\nu)$.
It can be seen that $\delta_\star(\alpha,\nu)$ is increasing with respect to the Poisson ratio, meaning that upon allowing shear interior deformations collapsing homologous elastic balls can form with smaller initial central densities than in the fluid case.

\begin{figure}[ht!]
\begin{center}
\includegraphics[width=0.7\textwidth,trim=0cm 0cm 0cm 0cm, clip ]{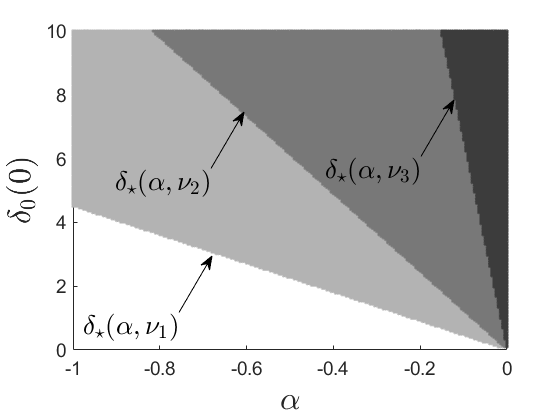}
\caption{The region in the $(\alpha,\delta_0(0))$ plane for which it found numerically that collapsing homologous self-gravitating polytropic elastic balls with zero shear at the boundary exist. These balls exist only for $\delta_0(0)\geq\delta_\star(\alpha,\nu)$; in this picture $\nu_1=0$, $\nu_2=0.25$, $\nu_3=0.45$. }
\label{homofig}
\end{center}
\end{figure}

In the following simple final theorem the existence of expanding homologous self-gravitating polytropic elastic balls is proved for some values of the shear parameter $\beta$.

\begin{theorem}\label{mainhomo}
If $\alpha>0$ and $1<\beta\leq4/3$, then for all $\delta_0^c>0$ there holds $Z_\mathrm{max}<\infty$,  $0<y_0(z)<1$, for $z\in (0,Z_\mathrm{max})$, and $y_0(z)\to 0^+$ as $z\to Z_\mathrm{max}^-$.  In this case there exists a unique continuously expanding homologous self-gravitating polytropic elastic ball $(\rho,p_\mathrm{rad},p_\mathrm{tan},u)$ of type $\mathrm{A}$ with central density $\rho(0,t)=\mathcal{K}\delta_0^c/\omega(t)$ and radius $R(t)= Z\omega(t)$, where $Z\in (0,Z_\mathrm{max})$ is the first solution of $y(Z)=y_\mathrm{b}$. Moreover $p_\mathrm{tan}(t,r)\geq p_\mathrm{rad}(t,r)$ holds for all times in the interior of the ball.
\end{theorem}
\begin{proof}
To start with we observe that by the exact same argument used in the proof of Lemma~\ref{propostatic}(i), strongly regular solutions of~\eqref{deltahom} satisfy
\[
\delta_0''(0)=-\frac{4\pi G}{3}\mathcal{K}^2\frac{\delta_0^c(\delta_0^c+\alpha)}{a(\delta_0^c,\delta_0^c)}=-\theta(\delta_0^c)^{2/3}(\delta_0^c+\alpha),\quad \eta_0''(0)=\frac{3}{5}\delta''(0).
\]
Hence, for $\alpha>0$, and by the exact same proof as in Lemma~\ref{propostatic}(ii), the inequalities $\eta_0(z)>\delta_0(z)$ and $\eta_0'(z)<0$ hold for $z\in (0,Z_\mathrm{max})$. Thus for $\alpha>0$ and $\beta\leq 4/3$ we have
\[
\delta_0'\leq -\left(\frac{\delta_0}{\eta_0}\right)^{1-\beta}\theta z\frac{\delta_0}{\eta_0^{1/3}}(\eta_0+\alpha)\leq -c\alpha z\delta^{2-\beta},
\] 
for a positive constant $c$. It follows that for $\alpha>0$ and $1<\beta\leq 4/3$, there hold $Z_\mathrm{max}<\infty$ and 
\[
\lim_{z\to Z_\mathrm{max}^-}\delta_0(z)=\lim_{z\to Z_\mathrm{max}^-}y_0(z)=0.
\]
As $y_\mathrm{b}\in (0,1)$, there exists a unique $Z\in (0,Z_\mathrm{max})$ such that $y(z)>y_\mathrm{b}$ for $z\in [0,Z_\mathrm{max})$ and $y(Z)=y_\mathrm{b}$, or equivalently $F_\mathrm{rad}(\delta_0(z),\eta_0(z))>0$ for $z\in (0,Z)$ and $F_\mathrm{rad}(\delta_0(Z),\eta_0(Z))=0$. Moreover, by the Baker-Ericksen inequality, $F_\mathrm{tan}(\delta_0(z),\eta_0(z))\geq F_\mathrm{rad}(\delta_0(z),\eta_0(z))$, hence defining $\rho(t,r)=\mathcal{K}\delta(t,r)$ and $(p_\mathrm{rad}(t,r),p_\mathrm{tan}(t,r))$ as in~\eqref{label}, we obtain that the quadruple $(\rho,p_\mathrm{rad},p_\mathrm{tan},u)\mathbb{I}_{r<R(t)}$ is a homologous self-gravitating elastic ball with radius $R(t)=Z\omega(t)$, for all $t\geq 0$, and $R(t)\to\infty$ as $t\to\infty$.
\end{proof}

\appendix

\section{Appendix: Proof of Theorem~\ref{maintheostatic}}

Recall that 
a solution of~\eqref{deltastat} in the interval $[0,R)$ is called regular if  $\delta\in C^0([0,R))\cap C^1((0,R))$ and $\delta(r)>0$, for all $r\in [0,R)$, and strongly regular if in addition $\delta\in C^1([0,R))$ and $\lim_{r\to 0^+}\delta(r)=0$; see Section~\ref{remstatic}. The following lemma contains results on the limit $r\to R_\mathrm{max}^-$ of regular solutions defined up to the maximal radius $R_\mathrm{max}$.
\begin{lemma}\label{existencelemma}
Let $\gamma\in\R$, $\beta\neq 0$ and $\delta_c>0$. Assume that there exists a unique regular solution of~\eqref{deltastat} satisfying $\eta(0)=\delta(0)=\delta_c>0$ and defined in a maximal radius interval $[0,R_\mathrm{max})$, $R_\mathrm{max}>0$.
Then the following holds:
\begin{itemize}
\item[(i)] If $R_\mathrm{max}<\infty$, then $\lim_{r\to R_\mathrm{max}^-}\delta(r)=0$ and $\lim_{r\to R_\mathrm{max}^-}\eta(r)>0$;
\item[(ii)] If $R_\mathrm{max}=\infty$, then $\lim_{r\to \infty}\delta(r)=\lim_{r\to\infty}\eta(r)=0$.
\end{itemize}
\end{lemma}
\begin{proof}

%
The inequalities $\delta(r)<\eta(r)$ and $\eta'(r)<0$, for $r\in (0,R_\mathrm{max})$, will be used in the following arguments, see Lemma~\ref{propostatic}.
\begin{itemize}
\item[(i)] We begin by showing that $\delta(r)$ has a (finite) limit when $r\to R_\mathrm{max}^-$. 
This is obvious when $\beta\leq\gamma$, as in this case $\delta(r)$ is non-increasing. Let $\beta>\gamma$. 
For $\beta\leq 2$ we write the $\delta$-equation in~\eqref{deltastat} as
\begin{equation}\label{temporalino}
\delta'=\frac{3(\beta-\gamma)}{r}\widetilde{B}(\delta/\eta)\eta-\theta r \eta^{1+\beta-\gamma}\delta^{2-\beta},
\end{equation}
where $\widetilde{B}(y)=B(y)y^{1-\beta}$ is bounded for $y\in [0,1]$. Hence for $r\in (0,R_\mathrm{max})$, there holds
\[
|\delta'(r)|\leq \frac{c_1}{r}+c_2 r,
\]
where $c_1,c_2$ are positive constants. It follows that $|\delta'|$ is bounded in the interval $(0,R_\mathrm{max})$ and 
thus $\delta(r)$ converges as $r\to R_\mathrm{max}^-$ when $\beta\leq 2$. For $\beta>2$ we instead write the $\delta$-equation as
\[
\frac{1}{\beta}(\delta^{\beta})'=\delta^{\beta-1} \delta'=\frac{3(\beta-\gamma)}{r}B(\delta/\eta)\eta^{\beta}-\theta r\eta^{1+\beta-\gamma}\delta
\]
and since $B(y)$ is bounded on $[0,1]$ for $\beta>2$, we obtain as before that $|(\delta^{\beta})'|$ is bounded on $(0,R_\mathrm{max})$. Hence $\delta^{\beta}(r)$ converges as $r\to R_\mathrm{max}^-$, and so $\lim_{r\to R_\mathrm{max}^-}\delta(r)$ exists for all $\beta\neq0$.
It will now be proved that this limit is zero. Assume, by contradiction, that $\lim_{r\to R_\mathrm{max}^-}\delta(r)>0$. Then there exists $c\in (0,1)$ such that $c<\delta(r)/\eta(r)<1$, for all $r\in (0,R_\mathrm{max}$), which implies that the right hand side of~\eqref{deltastat} is bounded on $(0,R_\mathrm{max})$.  But then $|\delta'|$ is bounded on $(0,R_\mathrm{max})$, contradicting the hypothesis that $R_\mathrm{max}<\infty$. 
Hence $\lim_{r\to R_\mathrm{max}^-}\delta(r)=0$ and 
\[
\lim_{r\to R_\mathrm{max}^-}\eta(r)=\frac{3}{R_\mathrm{max}^3}\int_0^{R_\mathrm{max}}\delta(s) s^2\,ds>0,
\]
must hold when $R_\mathrm{max}<\infty$. 
\item[(ii)]  Let $\eta_\infty=\lim_{r\to\infty}\eta(r)$. If $\eta_\infty=0$, then $\lim_{r\to\infty}\delta(r)=0$ and thus the claim follows. Assume $\eta_\infty>0$; in particular $\eta_\infty<\eta(r)<\eta(0)=\delta_c$, for all $r\in (0,\infty)$.  As in (i) we start by showing that $\lim_{r\to\infty}\delta(r)$ exists, which again is obvious for $\beta\leq\gamma$. Let $\beta>\gamma$ and assume first $\beta\geq1$. In this case we rewrite the $\delta$-equation in~\eqref{deltastat} as
\[
\delta'=\left(\frac{\delta}{\eta}\right)^{1-\beta}\eta^{2-\gamma}\delta\left(\frac{3(\beta-\gamma)}{r}\frac{B(\delta/\eta)}{\delta/\eta}\eta^{\gamma-2}-\theta r\right)\leq\left(\frac{\delta}{\eta}\right)^{1-\beta}\eta^{2-\gamma}\delta (c r^{-1}-\theta r), 
\]
where $c$ is a positive constant and where we used that $B(y)/y$ is bounded for $y\in (0,1)$ and $\beta\geq1$. Hence $\delta'(r)<0$ for $r>\sqrt{c/\theta}$ and thus $\lim_{r\to\infty}\delta(r)$ exists when $\beta\geq 1$. Assume now $\beta<1$. Then by~\eqref{temporalino} we can write 
\[
\delta'(r)=\frac{3(\beta-\gamma)}{r}\frac{\widetilde{B}(\delta/\eta)}{\delta/\eta}\delta-\theta r\eta^{1+\beta-\gamma}\delta^{2-\beta}\leq \frac{c_1}{r}\delta-c_2 r\delta^{2-\beta},
\]
where $c_1,c_2$ are positive constants and where we used that $\widetilde{B}(y)/y=B(y)y^{-\beta}$ is bounded for $\beta<1$. Hence $\delta(r)\leq u(r)$ for $r\geq 1$, where 
\[
u(r)=\left(\frac{c_2(1-\beta)}{2+c_1(1-\beta)}r^2+\frac{\delta(1)^{\beta-1}(2+c_1(1-\beta))-c_2(1-\beta)}{2+c_1(1-\beta)}\frac{1}{r^{c_1(1-\beta)}}\right)^{-1/(1-\beta)}
\]
is the solution of $u'(r)=c_1u(r)/r-c_2r u(r)^{2-\beta}$ with $u(1)=\delta(1)$. We conclude that $\delta(r)\to 0$ as $r\to\infty$ when $\beta<1$. We claim that $\delta(r)$ converges to zero at infinity even when $\beta>1$. Assume $\lim_{r\to\infty}\delta(r)=\delta_\infty>0$; then
\[
\delta'(r)\sim\left(\frac{\delta_\infty}{\eta_\infty}\right)^{1-\beta}\left(\frac{3(\beta-\gamma)}{r}B(\delta_\infty/\eta_\infty)\eta_\infty-\theta r\eta_\infty^{2-\gamma}\delta_\infty\right)\to -\infty,\quad\text{ as $r\to\infty$}, 
\]
a contradiction. Hence $\lim_{r\to\infty}\delta(r)=0$ must hold and thus also $\lim_{r\to\infty}\eta(r)=0$, by L'H\^{o}pital's rule.
\end{itemize}
\end{proof}
In the next proposition the existence of strongly regular solutions to~\eqref{deltastat} in a maximal radius interval $[0,R_\mathrm{max})$ is proved and sufficient conditions on the parameters $\gamma,\beta$ are given such that $R_\mathrm{max}$ is finite or $R_\mathrm{max}=\infty$. 
\begin{proposition}\label{teoremone}
For all $\gamma\in\R$, $\beta\neq 0$ and $\delta_c>0$, the system~\eqref{deltastat} admits a unique strongly regular local solution $(\delta,\eta)$ such that $\delta(0)=\eta(0)=\delta_c$. Letting $[0,R_\mathrm{max})$ be the maximal radius interval of definition of this solution the following holds:
\begin{itemize}
\item[$\mathrm{(A)}$] If $\gamma>2$ and $1<\beta\leq\gamma$, then $R_\mathrm{max}<\infty$; 
\item[$\mathrm{(B)}$] If $0<\gamma\leq \beta<1$ or $\beta<\gamma\leq 1$, then $R_\mathrm{max}=\infty$ and $\lim_{r\to\infty}y(r)=\frac{4-3\gamma}{3(2-\gamma)}$. 
\end{itemize} 
\end{proposition}
\begin{proof}
Assume first $\gamma\neq2$. The change of variables
\[
y=\delta/\eta,\quad v=\theta r^2\eta^{2-\gamma}y^{1-\beta},\quad \xi=\log r
\]
transforms~\eqref{deltastat} into the autonomous dynamical system
\begin{subequations}\label{dynsys}
\begin{align}
&\frac{dy}{d\xi}=[\Upsilon(y)-v]y,\\
&\frac{dv}{d\xi}=[(1-\beta)(\Upsilon(y)-v)+2-3(2-\gamma)(1-y)]v,
\end{align}
where
\begin{equation}\label{upsilon}
\Upsilon(y)=3(1-y)+3(\beta-\gamma)B(y)y^{-\beta}.
\end{equation}
\end{subequations}
As $(dy/d\xi)_{y=1}<0$, for all $v>0$, the open region $\mathcal{U}=\{v>0, y<1\}$ of the state space is future invariant (which is equivalent to the bound $\eta>\delta$ proved in Lemma~\ref{propostatic}(ii)). A simple local stability analysis shows that
the boundary point $\mathrm{O}$ with coordinates $v=0, y=1$ is an hyperbolic saddle. The positive eigenvalue of the linearised flow around $\mathrm{O}$ is equal to 2 and the corresponding eigenvector $-5\partial_y+\partial_v$ points toward the interior of $\mathcal{U}$. Hence there exists exactly one orbit $\Gamma=(\Gamma_y,\Gamma_v)\subset\mathcal{U}$ such that $\lim_{\xi\to-\infty}\Gamma(\xi)=\mathrm{O}$. Moreover $1-y(\xi)\sim 5C e^{2\xi}$ and $v(\xi)\sim C e^{2\xi}$ as $\xi\to-\infty$ along this orbit, where $C$ is a positive constant. This orbit corresponds to a one parameter family of regular solutions of~\eqref{deltastat} up to $\xi_*$ such that $\Gamma_y(\xi)>0$ for $\xi<\xi_*$, and by Lemma~\ref{propostatic}(i) these solutions are strongly regular. Uniqueness follows if we show that the constant $C$ is uniquely determined by the center datum $\delta(0)=\eta(0)=\delta_c$. 
As $\gamma\neq 2$, this follows by the definition of $v$, which gives $C=\theta \delta_c^{2-\gamma}$. Assume now $\gamma=2$. In this case, the system~\eqref{deltastat} is equivalent to the following decoupled system on $y=\delta/\eta$ and $\eta$:
\begin{equation}\label{tempocacchio}
y'=\left(\frac{\Upsilon(y)}{r}-\theta r y^{1-\beta}\right)y,\quad \eta'=-\frac{3\eta}{r}(1-y).
\end{equation}
The same argument as above gives now a local unique regular solution of the $y$-equation with asymptotic behavior $y=1+O(r^2)$ as $r\to 0^+$, and thus uniqueness for the system~\eqref{tempocacchio} follows by simply integrating the $\eta$-equation with center datum $\eta(0)=\delta_c$. By Lemma~\ref{propostatic} the unique regular solution is strongly regular and thus the first part of the proposition is proved. 

{\bf Proof of $\mathrm{(A)}$:}
As $\beta\leq\gamma$ we have
\begin{equation}\label{temporalaccio}
\delta'\leq -\theta r \eta^{1+\beta-\gamma}\delta^{2-\beta}.
\end{equation}
For $\beta\leq \gamma-1$ we use $\eta\leq\eta(0)=\delta_c$ in~\eqref{temporalaccio} to obtain $\delta'\leq -cr\delta^{2-\beta}$, and thus
\[
\delta(r)\leq\left(\delta_c^{\beta-1}-\frac{c}{2}(\beta-1)r^2\right)^{\frac{1}{\beta-1}},
\] 
which, as $\beta>1$, implies $R_\mathrm{max}<\infty$. For $\gamma-1<\beta\leq \gamma$ we instead use $\eta>\delta$ in~\eqref{temporalaccio} to obtain $\delta'\leq-\theta r\delta^{3-\gamma}$, which, as $\gamma>2$, implies again $R_\mathrm{max}<\infty$.

{\bf Proof of $\mathrm{(B)}$:}
To prove $\mathrm{(B)}$ we study the qualitative behavior toward the future of the orbit $\Gamma$ of the dynamical system~\eqref{dynsys} originating from the fixed point $\mathrm{O}$. 
As $\beta<1$, the function $\Upsilon$ is bounded for $y\in[0,1]$ and 
\[
\Upsilon(0)=3\frac{1-\gamma}{1-\beta}\geq 0.
\] 
In fact, a straightforward analysis shows that $\Upsilon(y)>0$ for all $y\in (0,1)$.
In particular the region $\mathcal{V}=\{v>0, 0<y<1\}$ is future invariant and $\Gamma(\xi)\subset\mathcal{V}$, for all $\xi\in\R$. Moreover $v'(\xi)\leq (a-(1-\beta)v)v$, where $a=\sup_{y\in(0,1)}[(1-\beta)\Upsilon(y)+2-3(2-\gamma)(1-y)]<\infty$ and thus $v(\xi)\leq a/(1-\beta)$ along the orbit $\Gamma$. It follows that the $\omega$-limit set $\omega(\Gamma)$ of $\Gamma$ is not empty. By Poincar\'e-Bendixson theorem, $\omega(\Gamma)$ must be one of the following sets: (1) a fixed point; (2) a periodic orbit; (3)
 a connected set consisting of a finite number of fixed points $\{P_1,\dots,P_n\}$ together with homoclinic and heteroclinic orbits connecting $P_1,\dots,P_n$.
Let $\phi(y,v)=v^{-1}y^{\beta-2}$ and let $F(v,y)$ denote the vector field in the right hand side of~\eqref{dynsys}. Since $\nabla\cdot(\phi F)(y,v)=-3\phi(y,v)(1-\gamma(1-y))$ is negative for $\gamma\leq 1$ and $(y,v)\in\mathcal{V}$, then, by Dulac-Bendixson theorem, no periodic orbits exist in the region $\mathcal{V}$ and thus the alternative (2) above is not possible.
The alternative (3) can be ruled out by studying the stability properties of the fixed points of the flow. Besides $\mathrm{O}$, the dynamical system~\eqref{dynsys} admits the fixed points $\mathrm{Q}=(0,0)$ and
\[
\mathrm{P}=(\Upsilon(y_\mathrm{P}),y_\mathrm{P}),\quad \text{where}\quad \ y_\mathrm{P}=\frac{4-3\gamma}{3(2-\gamma)}.
\]
No other fixed points are present when $\gamma\leq 1$. A simple local stability analysis shows that $\mathrm{P}$ is an hyperbolic sink, and thus it is a local attractor for a one parameter family of interior orbits, 
while $\mathrm{Q}$ is an hyperbolic saddle. The stable manifold of $\mathrm{Q}$ is tangent to the axis $y=0$, while the unstable manifold is tangent to $v=0$. In particular there is no interior orbit which converges to or emanates from $\mathrm{Q}$. Putting this information together we conclude that none of the structures mentioned in the alternative (3) above exists in the region $\mathcal{V}$. Thus Poincar\'e Bendixson theorem entails that the fixed point $\mathrm{P}$ is the $\omega$-limit set of all orbits entering the region $\mathcal{V}$ and so in particular $\Gamma(\xi)\to\mathrm{P}$ as $\xi\to\infty$, which is equivalent to $\lim_{r\to\infty}y(r)=\frac{4-3\gamma}{3(2-\gamma)}$.  

\end{proof}


\begin{proof}[Proof of Theorem~\ref{maintheostatic}]
When~\eqref{condgammabetaI} holds, Proposition~\ref{teoremone} gives that the maximal radius interval of definition of strongly regular solutions of~\eqref{deltastat} is finite  
and 
by  Lemma~\ref{existencelemma}(i)
\[
y(r)=\frac{\delta(r)}{\eta(r)}\to 0,\quad \text{as $r\to R_\mathrm{max}^-$}.
\]
As $y(0)=1$ and $y_\mathrm{b}\in (0,1)$, there exists a unique $R\in (0,R_\mathrm{max})$ such that $y(r)>y_\mathrm{b}$ for $r\in [0,R)$ and $y(R)=y_\mathrm{b}$. Thus, letting $p_\mathrm{rad}(r)=F_\mathrm{rad}(\delta(r),\eta(r))$, we have $p_\mathrm{rad}(R)=0$ and $p_\mathrm{rad}(r)>0$, for $r\in [0,R)$. Moreover, by Lemma~\ref{propostatic}(iii), 
\[
p_\mathrm{tan}(r)=F_\mathrm{tan}(\delta(r),\eta(r))\geq p_\mathrm{rad}(r), \quad r>0.
\]
Hence, letting $\rho(r)=\mathcal{K}\delta(r)$, we obtain that $(\rho,p_\mathrm{rad},p_\mathrm{tan})\mathbb{I}_{r\leq R}$ is a type $\mathrm{A}$  static self-gravitating elastic ball. When~\eqref{condgammabetaII} holds, strongly regular solutions of~\eqref{deltastat} are global and $y\to\frac{4-3\gamma}{3(2-\gamma)}=y_\infty>0$ as $r\to\infty$, see again Proposition~\ref{teoremone}. Since the last inequality in~\eqref{condgammabetaII} is equivalent to $y_\infty<y_\mathrm{b}$, then there exists $R\in (0,R_\mathrm{max})$ such that $y(r)>y_\mathrm{b}$ for $r\in [0,R)$ and $y(R)=y_\mathrm{b}$, and thus the proof can be completed as before.
\end{proof}

\end{document}